\documentclass[12pt]{article}
\usepackage{amsmath, amsthm, amssymb, amscd, mathptmx}

\usepackage{amsfonts}

\usepackage{color}
\usepackage{pstricks}

\usepackage[all]{xy}\CompileMatrices\SelectTips{cm}{12}

\theoremstyle{plain}
\newtheorem{Thm}{\sc Theorem}[section]
\newtheorem{Theorem}[Thm]{\sc Theorem}
\newtheorem{Corollary}[Thm]{\sc Corollary}

\newtheorem*{Corollary*}{\sc Corollary}

\newtheorem{Proposition}[Thm]{\sc Proposition}
\newtheorem*{Proposition*}{\sc Proposition}
\newtheorem{Lemma}[Thm]{\sc Lemma}

\theoremstyle{definition}
\newtheorem{Definition}[Thm]{Definition}

\theoremstyle{remark}
\newtheorem{Remark}[Thm]{Remark}
\newtheorem{Example}[Thm]{Example}
\newtheorem*{Example*}{Example}
\newtheorem*{Remark*}{Remark}

\renewcommand{\AA}{{\mathbb A}}

\newcommand{\ZZ}{{\mathbb Z}}

\newcommand{\PP}{{\mathbb P}}
\newcommand{\QQ}{{\mathbb Q}}
\newcommand{\RR}{{\mathbb R}}

\newcommand{\cA}{{\mathcal A}}
\newcommand{\cB}{{\mathcal B}}
\newcommand{\cD}{{\mathcal D}}
\newcommand{\cE}{{\mathcal E}}
\newcommand{\cF}{{\mathcal F}}
\newcommand{\cG}{{\mathcal G}}
\newcommand{\cH}{{\mathcal H}}
\newcommand{\cL}{{\mathcal L}}
\newcommand{\cM}{{\mathcal M}}

\newcommand{\cO}{{\mathcal O}}
\newcommand{\cS}{{\mathcal S}}
\newcommand{\cT}{{\mathcal T}}

\newcommand{\cX}{{\mathcal X}}
\newcommand{\cY}{{\mathcal Y}}

\newcommand{\Pic}{{\mathop{\rm Pic \, }}}

\newcommand{\Hig}{{\mathop{\operatorname{HIG}\, }}}

\newcommand{\Coh}{{\mathop{\operatorname{Coh}\, }}}
\renewcommand{\Ref}{{\mathop{\operatorname{Ref}\, }}}
\newcommand{\Mic}{{\mathop{\operatorname{MIC}\, }}}
\newcommand{\Mod}[1]{\mathop{\operatorname{{#1}-Mod}\, }}
\newcommand{\refMod}[1]{\mathop{\operatorname{{#1}-Mod}^{\mathop{\rm ref}}\, }}

\newcommand{\Hom}{{\mathop{{\rm Hom}}}}

\newcommand{\ch}{{\mathop{\rm ch \, }}}
\newcommand{\cHom}{{\mathop{{\cal H}om}}}

\newcommand{\Gr}{{\mathop{Gr}}}
\newcommand{\id}{\mathop{\rm Id}}

\newcommand{\ad}{{\mathop{\rm ad \, }}}
\newcommand{\can}{{\mathop{\rm can}}}

\newcommand{\End}{{\mathop{{\mathcal E}nd}}}

\newcommand{\rk}{{\mathop{\rm rk \,}}}

\newcommand{\Sym}{{\mathop{{\rm Sym}}}}

\newcommand{\Spec}{{\mathop{{\rm Spec\, }}}}

\newcommand{\refl}{{\mathop{\operatorname{\rm ref} }}}

\newcommand{\sslash}{\mathbin{
		\mathchoice{/\mkern-6mu/}% \displaystyle
		{/\mkern-6mu/}% \textstyle
		{/\mkern-5mu/}% \scriptstyle
		{/\mkern-5mu/}}}% \scriptscriptstyle

\begin{document}

\markboth {\rm }{}

\title{Bogomolov's inequality and Higgs sheaves on normal varieties in positive characteristic}
\author{Adrian Langer} \date{\today}

%%%
\maketitle
%%%

%%%%%%%%%%%%%%%%%%%%%%%%%%%%%%%%%%

{\noindent \sc Address:}\\
Institute of Mathematics, University of Warsaw,
ul.\ Banacha 2, 02-097 Warszawa, Poland\\
e-mail: {\tt alan@mimuw.edu.pl}

\medskip

\begin{abstract}
We prove Bogomolov's inequality on a normal projective  variety in positive characteristic and we use it to show some new restriction theorems and a new boundedness result. Then we redefine Higgs sheaves on normal varieties and we prove restriction theorems and Bogomolov type inequalities for semistable logarithmic Higgs sheaves on some normal varieties in an arbitrary characteristic. 
\end{abstract}

\section*{Introduction}

Let $X$ be a smooth projective variety of dimension $n$ defined over an algebraically closed field $k$  and let $H$ be an ample line bundle on $X$. Let $\cE$ be a coherent torsion free $\cO_X$-module of rank $r$ and let us set  $\Delta (\cE):=2r c_2 (\cE)-(r-1)c_1 (\cE)^2$.

A well known Bogomolov's theorem says that if $k$ has characteristic zero and  $\cE$ is slope $H$-semistable then 
$$\int_X \Delta (\cE) H^{n-2}\ge 0.$$
This theorem was first proven by F. Bogomolov in the surface case. The higher dimensional case  follows from restriction theorems for semistability (e.g., one can use the Mehta--Ramanathan restriction theorem).
An analogue of this theorem for slope $H$-stable Higgs bundles was proven by C. Simpson in \cite{Si} using analytic methods. Simpson's paper contains also  applications of this result to uniformization and the Miyaoka --Yau  inequality in higher dimensions (although only in the non-logarithmic case). 
Later, T. Mochizuki in \cite{Mo} generalized this inequality to the logarithmic case (he also used analytic methods). An algebraic proof of Bogomolov's inequality for Higgs sheaves 
appeared in \cite{La3} and in the logarithmic case in \cite{La4}. These papers contained also generalization of these results to positive characteristic.

More recently, in characteristic zero the above results have been generalized in \cite{GKPT2}
to projective varieties with klt singularities (but not in the logarithmic case). In the  mildly singular logarithmic case one also knows the Miyaoka--Yau inequality (see \cite[Chapter 10]{Ko} and \cite{La0} for the $2$-dimensional case and \cite{GT} for higher dimensions).

\medskip

One of the main motivations behind this paper is generalization of the above results to positive characteristic and strengthening of the results known in the characteristic zero. We also deal with semistability defined by a collection of nef line bundles instead of one ample line bundle. An importance of considering this generalized situation was first recognized by
Y. Miyaoka in \cite{Mi}, who proved Bogomolov's inequality for torsion free sheaves on normal 
varieties smooth in codimension $2$ in case of collections of ample and one nef line bundles.
However, it is not completely clear to the author if the original proof of Mehta--Ramanathan's restriction theorem works so easily for multipolarizations on normal varieties as claimed in \cite[Corollary 3.13]{Mi}. In case of one ample line bundle on a normal projective variety defined over an algebraically closed field of characteristic zero, restriction theorem for semistable sheaves has been proven by H. Flenner in \cite{Fl}. However, it seems that his proof cannot be generalized to multipolarizations. 

In case of smooth projective varieties the corresponding Bogomolov's inequality in any characteristic was proven in \cite{La1} (however, the proof uses a different, new restriction theorem).
Using resolution of singularities one can use this to obtain Mehta--Ramanathan's restriction theorem 
for multipolarizations on normal varieties in characteristic zero. 
In case of Higgs sheaves on smooth projective varieties, restriction theorem and Bogomolov's inequality for multipolarizations has been proven in \cite{La3} and in the logarithmic case in \cite{La4}.

\medskip

Our first main result is a strong restriction theorem for multipolarized normal varieties in positive characteristic, analogous to \cite[Theorem 5.2 and Corollary 5.4]{La1}. One of the problems here is with the definition of Chern classes. Below we use Chern classes of reflexive sheaves defined in \cite{La-Chern} (see Subsection \ref{Subsection:intersection-theory} for a few basic properties). 

Let $X$ be a normal projective variety of dimension $n$ defined over an algebraically closed field $k$ of characteristic $p>0$. Let us fix a collection $(L_1,...,L_{n-1})$ of ample line bundles on $X$ 
(in fact, we usually need weaker assumptions  on this collection).  Let  
us set  $d=L_1^2L_2...L_{n-1}$. Then we have the following 
result (see Subsection \ref{Subsection:beta} for the definition of $\beta_r$).

\begin{Theorem}\label{Main1}
Let  $\cE$ be a coherent reflexive $\cO_X$-module of rank $r\ge 2$. Let  $m$ be an integer such that
$$m>\left\lfloor \frac{r-1}{ r}\int_X\Delta (\cE)L_2\dots L_{n-1} +\frac{1}{dr(r-1)}
+\frac{(r-1)\beta _r}{dr}\right\rfloor $$
and let $H\in |L_1^{\otimes m}|$  be a normal hypersurface.
 If $\cE$ is slope $(L_1,...,L_{n-1})$-stable 	then $\cE|_H$ is  slope $(L_2|_H,\dots ,L_{n-1}|_H)$-stable.
\end{Theorem}

The above theorem implies the following boundedness result. 

\begin{Theorem} \label{Main2}
	Let us fix some positive integer $r$, integer $\ch_1$ and some real
	numbers $\ch _2$ and $\mu _{\max}$. 
	Then the set of reflexive coherent $\cO_X$-modules $\cE$ of rank $r$ with $\int _X \ch_1 (\cE) L_1...L_{n-1}=\ch_1$, $\int _X \ch_2 (\cE) L_1...\widehat{L_i}...L_{n-1} \ge \ch_2$ for $i=1,...,n-1$,  and $\mu _{\max } (\cE)\le \mu _{\max}$ is bounded.
\end{Theorem}

In the statement above it is not even clear that the number of Hilbert polynomials of sheaves in the considered family is finite. In case $L_1=...=L_{n-1}$ the above theorem follows from \cite[Theorem 4.4]{La1}. If $X$ is smooth then the above theorem follows from \cite[Corollary 5.4]{La1}. But
it is no longer the case if we consider multipolarizations on normal varieties.

\medskip

We also prove an analogue of Theorem \ref{Main1}  for  (semi)stable Higgs sheaves on normal varieties (see Theorem \ref{restriction-for-Higgs} for a more precise version).

\begin{Theorem}	\label{Main5}
	Let  $D\subset X$ be an effective reduced Weil divisor and let $(\cE , \theta )$ be a reflexive logarithmic Higgs sheaf of rank $r\ge 2$ on $(X,D)$. 
Let $m_0$ be a non-negative integer such that $T_X (\log \, D)\otimes L_1^{\otimes m_0}$ is globally generated. Let   $m$ be an integer such that
	$$m>\max \left(\left\lfloor \frac{r-1}{ r}\int_X\Delta (\cE)L_2\dots L_{n-1} +\frac{1}{dr(r-1)}
	+\frac{(r-1)\beta _r}{dr}\right\rfloor , 2(r-1)m_0^2 \right).$$
and let $H\in |L_1^{\otimes m}|$ be a general divisor. If  $(\cE , \theta )$ is slope $(L_1,...,L_{n-1})$-stable then the logarithmic Higgs sheaf $(\cE , \theta )|_H$ on $(H,D\cap H)$ is  slope $(L_2|_H,\dots ,L_{n-1}|_H)$-stable.
\end{Theorem}

This theorem generalizes \cite[Theorem 10]{La3} that works for smooth varieties liftable to $W_2(k)$. 

\medskip 

Finally, we use the above results to prove the following Bogomolov's inequality for reflexive Higgs sheaves on mildly singular normal varieties. Note that, unlike previously known results for singular varieties in characteristic zero, our theorem holds for log pairs.

\begin{Theorem} \label{Main3} 
	Let $D\subset X$ be an effective reduced Weil divisor such that the pair $(X,D)$ is almost liftable to $W_2(k)$ and it has  $F$-liftable singularities in codimension $2$. Then for any
	slope $(L_1,...,L_{n-1})$-semistable logarithmic reflexive Higgs sheaf $(\cE, \theta)$ of rank
	$r\le p$ we have
	$$\int_X \Delta (\cE) L_2...L_{n-1}\ge 0.$$
\end{Theorem}

For the meaning of almost liftable log pair and  $F$-liftable singularities  we refer the reader to Definitions \ref{definition-F-liftable} and \ref{definition-F-liftable-2}. If $X$ is liftable to $W_2(k)$ then it is almost liftable to $W_2(k)$ and almost all reductions of varieties from characteristic zero satisfy this condition. 
To understand the second notion we note that a reduction of  quotient surface singularity is $F$-liftable in large characteristics (see Subsection \ref{F-liftability}). In fact, for a dense set of primes, reductions of surfaces with log canonical singularities have $F$-liftable singularities (we do not prove this non-trivial fact as we will not need it in the following).

\medskip

Now let $X$ be a normal projective variety of dimension $n$ defined over an algebraically closed field $k$ of characteristic $0$. Assume that $X$ has at most quotient singularities in codimension $2$.
Let  us fix a collection $(L_1,...,L_{n-1})$ of ample line bundles on $X$ and  set  $d=L_1^2L_2...L_{n-1}$.   Then Theorem \ref{Main1} implies the following restriction theorem:

\begin{Theorem}\label{Main7}
	Let  $\cE$ be a coherent reflexive $\cO_X$-module of rank $r\ge 2$. Let  $m$ be an integer such that
	$$m>\left\lfloor \frac{r-1}{ r}\int_X\Delta (\cE)L_2\dots L_{n-1} +\frac{1}{dr(r-1)}
	\right\rfloor $$
	and let $H\in |L_1^{\otimes m}|$  be a normal hypersurface.
	If $\cE$ is slope $(L_1,...,L_{n-1})$-stable 	then $\cE|_H$ is  slope $(L_2|_H,\dots ,L_{n-1}|_H)$-stable.
\end{Theorem}

Now let  us also fix an effective reduced Weil divisor $D\subset X$ such that the pair $(X,D)$ is log canonical in codimension $2$. Theorem \ref{Main5}  implies the following result:

\begin{Theorem}	\label{Main6}
Let  $(\cE , \theta )$ be a reflexive logarithmic Higgs sheaf of rank $r\ge 2$ on $(X,D)$.
Let $m_0$ be a non-negative integer such that $T_X (\log \, D)\otimes L_1^{\otimes m_0}$ is globally generated. Let   $m$ be an integer such that
$$m>\max \left(\left\lfloor \frac{r-1}{ r}\int_X\Delta (\cE)L_2\dots L_{n-1} +\frac{1}{dr(r-1)}
\right\rfloor , 2(r-1)m_0^2 \right).$$
and let $H\in |L_1^{\otimes m}|$ be a general divisor. If  $(\cE , \theta )$ is slope $(L_1,...,L_{n-1})$-stable then the logarithmic Higgs sheaf $(\cE , \theta )|_H$ on $(H,D\cap H)$ is  slope $(L_2|_H,\dots ,L_{n-1}|_H)$-stable.
\end{Theorem}

Theorem \ref{Main6} generalizes \cite[Theorem 5.22]{GKPT2} and \cite[Theorem 6.1]{GKPT1}, which are non-effective. In fact, we prove a stronger version of Theorem \ref{Main6} (see Theorem \ref{restriction-for-Higgs-char-0}) that works for all normal divisors $H$ for which  restriction of $(\cE , \theta )$ to $H$ gives a  logarithmic Higgs sheaf  on $(H,D\cap H)$.

Similarly, Theorem  \ref{Main3} can be used to prove the following inequality generalizing \cite[Theorem 6.1]{GKPT2}.

\begin{Theorem} \label{Main4} 
For any slope $(L_1,...,L_{n-1})$-semistable logarithmic reflexive Higgs sheaf $(\cE, \theta)$ we have
	$$\int_X \Delta (\cE) L_2...L_{n-1}\ge 0.$$
\end{Theorem}

Note that Bogomolov's inequality for logarithmic Higgs sheaves has not been known so far even 
on klt pairs. As in \cite{Si} and \cite{GKPT2} the above theorem implies Miyaoka--Yau inequalities for singular log pairs. Here we show only some simple applications of Theorem \ref{Main3} to general Miyaoka--Yau inequalities in positive characteristic (see Section \ref{Section:Miyaoka-Yau}), leaving statement of general results in characteristic zero to the interested reader. Let us remark that unlike previous works on Chern number inequalities in higher dimensions (e.g., \cite{GKPT2} and \cite{GT})
our method should work in much more general situations in characteristic zero, the only obstacle being  unknown behaviour of Chern classes of reflexive sheaves under tensor operations on normal surfaces (see \cite{La-Chern}). In particular, an analogue of Theorems \ref{Main1}, \ref{Main2} and \ref{Main5}
should hold on any normal variety in characteristic zero and an analogue of Theorem \ref{Main3} should hold for any pair $(X,D)$ which is log canonical in codimension $2$. Appropriate versions are also expected if $D$ is an arbitrary effective Weil $\QQ$-divisor.

\medskip

Here we should  warn the reader that our definition of a reflexive Higgs sheaf is weaker than the one used in \cite{GKPT2} and \cite{GKPT1}. More precisely, a logarithmic reflexive Higgs sheaf $(\cE, \theta)$  in our sense is a pair consisting of a coherent reflexive $\cO_X$-module $\cE$ and an $\cO_X$-linear map $T_X (\log D)\otimes _{\cO_X} \cE\to \cE$ satisfying additional integrability condition. In the situation of \cite{GKPT2} this would correspond to considering 
$\cE\to (\cE\otimes_{\cO_X} \Omega _X)^{**}$ instead of $\cE\to \cE\otimes _{\cO_X}(\Omega _X)^{**}$.
This is  why in Section \ref{Section:Higgs-sheaves} we carefully explain differences between our approach and \cite{GKPT2}. D. Greb et al. use a different definition as they need to pullback Higgs sheaves by birational morphisms to pass to a resolution of singularities. On the other hand, they cannot take duals or pushforward Higgs sheaves by open embeddings as is allowed in our approach. In this paper we do not use  Kebekus's pullback functor for reflexive differentials on klt pairs (see \cite{Ke}) and we do not pullback Higgs sheaves by birational morphisms  (cf. Subsection \ref{pullback-char-0}).  This allows us to obtain stronger results, e.g., Bogomolov's inequality for reflexive extensions of semistable Higgs sheaves on the regular locus (cf. \cite[Theorem 6.1]{GKPT2}).

\medskip

Further applications of the obtained results to non-abelian Hodge theory and Simpson's correspondence are 
postponed to \cite{La9}.

\medskip

The structure of the paper is as follows. In the first section we gather some auxiliary results and introduce some notation. In Section 2 we prove a few results on Chern classes of reflexive sheaves on normal varieties in positive characteristic. These resuts are used in Section 3 to prove generalized versions of Theorems \ref{Main1} and \ref{Main2}. In Section 4 we study modules over Lie algebroids and generalized Higgs sheaves on normal varieties. In Section 5 we prove generalized versions of Theorems  \ref{Main5} and \ref{Main3}. In Section 6 we apply these results to obtain the Miyaoka--Yau inequality for some normal varieties in positive characteristic. In Section 7 we show some applications of the obtained results in characteristic zero, proving Theorems \ref{Main7}, \ref{Main6} and \ref{Main4}.
Section 8 contains an appendix in which we recall construction of the inverse Cartier  transform  used in Section 5.

\subsection*{Notation}

If $f: X\to Y$ is a morphism between normal schemes and $\cE$
is a coherent reflexive $\cO_Y$-module then we set
$$f^{[*]}\cE=(f^*\cE)^{**}.$$
If $f$ is flat then we have $f^{[*]}\cE= f^*\cE $.
But if $f$ is not flat then usually the canonical map $f^*\cE \to f^{[*]}\cE$ is not an isomorphism.

If $X$ is a normal scheme of characteristic $p$ then we denote by $F_X$ its absolute Frobenius morphism. If $\cE$ is a coherent reflexive $\cO_X$-module then for any positive integer $m$ we set
$$F_X^{[m]}\cE=(F_X^m)^{[*]}\cE .$$

\section{Preliminaries}

\subsection{Reflexive sheaves}\label{reflexive-sheaves}

In this subsection $X$ is an integral normal scheme, which is locally of finite type over a field $k$.
By  $\Ref(\cO_X)$  we denote the category of coherent reflexive $\cO_X$-modules. It is a full subcategory of the category $\Coh(\cO_X)$ of coherent  $\cO_X$-modules.  The inclusion functor $\Ref(\cO_X)\to \Coh (\cO_X)$ comes with a left adjoint $(\cdot )^{**}:\Coh(\cO_X)\to \Ref (\cO_X)$ given by the reflexive hull. The category $\Ref(\cO_X)$ 
comes with an associative and symmetric tensor product $\hat\otimes$ given by
$$\cE\hat\otimes \cF:=(\cE\otimes _{\cO_X} \cF)^{**}.$$
An open subset $U\subset X$ is called \emph{big}  if its complement $X\backslash U$  has codimension $\ge 2$ in $X$. If we consider $U$ as a subscheme of $X$ then we talk about a big open subscheme.
The following well-known lemma can be found in \cite[Lemma 0EBJ]{SP}.

\begin{Lemma}\label{restriction-to-open} 
	Let $j: U\hookrightarrow X $ be a big open subscheme.  Then  $j_*$ and $j^*$ define adjoint equivalences of categories $\Ref (\cO_X)$ and $\Ref(\cO_U)$. 
\end{Lemma}

\medskip

Since $X$ is normal, its regular locus $X_{\mathrm {reg}}\subset X$ is a big open subset.

\medskip

\begin{Lemma}\label{pull-back-of-dual}
Let $f: X\to Y$ be a finite dominant  morphism of integral locally Noetherian normal schemes. 
If $\cE$ is a coherent $\cO_Y$-module then we have a canonical isomorphism
	$$f^{[*]}(\cE ^*)\simeq (f^*\cE)^*.$$
\end{Lemma}

\begin{proof}	
We have a natural map 
$$f^*(\cE^*)=f^*\cHom_{\cO_Y} (\cE, \cO_Y)\to
	 \cHom_{\cO_X} (f^*\cE , f^*\cO _Y)= (f^*\cE)^*.$$
If $\cE$ is  torsion free then there exists a big open subset $V\subset Y$ such that $\cE _V$ is finite locally free.
Then the above map is an isomorphism on $U=f^{-1}(V)$.  This subset of $X$ is big because $f$ is finite and dominant. Since $(f^*\cE)^*$ is reflexive, this induces an isomorphism $f^{[*]}(\cE ^*)\simeq (f^*\cE)^*.$

If $\cE$ is not torsion free, then the pullback of the quotient map $\cE\to \tilde \cE:=\cE/  {\mathrm{Torsion}}$ is  surjective.
Hence  the dual map $( f^*\tilde \cE)^* \to (f^*\cE)^* $ is an isomorphism. But if we apply the lemma to $\tilde \cE $
 then we get  $f^{[*]}(\cE ^*)\simeq (f^*\tilde \cE )^*$, which proves the required assertion.
\end{proof}

\subsection{F-liftable schemes}\label{F-liftability}

A Weil divisor on a locally Noetherian integral scheme $X$ is a formal sum $\sum a_i D_i$,
where $a_i\in \ZZ$ and $D_i$ are prime divisors. However, if all nonzero $a_i$ are equal to $1$ then we can consider an effective reduced Weil divisor as a reduced induced scheme structure on  $D:=\bigcup_{\{ i: a_i\ne 0\}}  D_i\subset X$. If $f: X\to S$ is a morphism of schemes and $X$ is an integral locally Noetherian normal scheme then we say that a subscheme $D\subset X$ is a \emph{relative effective reduced  Weil divisor} on $X/S$ if $D$ is an effective reduced Weil divisor and $D\to S$ is a flat morphism.

\medskip

A \emph{log pair} $(X,D)$ is a pair consisting of a normal variety $X$ defined over a perfect field $k$ and an effective reduced Weil divisor $D$ on $X$ (we allow $D=0$). We say that $(X, D)$ is \emph{log smooth} if $X$ is smooth and $D$ is a normal crossing divisor. In this subsection we assume that $k$ has positive characteristic $p$. We also set $S=\Spec k$ and $\tilde S=\Spec W_2(k)$.

\begin{Definition}\label{definition-F-liftable}
	Let $(X, D)$ be a log pair and let us write $D=\sum D_i$, where $D_i$ are irreducible. 
	We say that $(X,D)$ is
	\begin{enumerate}
	\item \emph{liftable to $W_2(k)$} if there exists a flat morphism $\tilde X\to \tilde S$ and 
a relative effective Weil divisor $\tilde D=\sum \tilde D_i$ on $\tilde X/\tilde S$ such that $(X,D_i)= (\tilde X\times _{\tilde S}S , \tilde D_i\times _{\tilde S}S)$ for all $i$. 
Such a pair 
		$(\tilde X, \tilde D)$ is called a \emph{lifting} of $(X, D)$ to $W_2(k)$. 
	\item \emph{$F$-liftable} if there exists a lifting  $(\tilde X, \tilde D)$  of $(X, D)$ to $W_2(k)$
and a morphism $\tilde F_X: \tilde X\to \tilde X$ restricting to $F_X$ modulo $p$ such that
for each $D_i$ the image of $\tilde F_X^*I_{\tilde D_i}\to \cO_{\tilde X}$ is contained in  $ I_{\tilde D_i}^p$. In this case we say that $\tilde F_X$ \emph{is compatible with} $\tilde D$ and we call $\tilde F_X$ an \emph{$F$-lifting of $(X,D)$ (compatible with $(\tilde X, \tilde D)$)}.
\item \emph{almost liftable to $W_2(k)$}  if there exists a big open subset $U\subset X$ such that the pair $(U, D_U=D\cap U)$ is liftable to $W_2(k)$. The corresponding lifting of $(U, D_U)$ is called an \emph{almost lifting} of $(X,D)$.
	\item \emph{almost $F$-liftable} if there exists a big open subset $U\subset X$ such that the pair
	$(U, D_U)$ is $F$-liftable. The corresponding lifting  is called an \emph{almost $F$-lifting} of $(X,D)$.	
\end{enumerate}
\end{Definition}

\begin{Remark}
	\begin{enumerate}
		\item 
		If $U\subset X$ as in (3)-(4) exists then we can always find a big open subset $V\subset X$
		such that $(V, D_V)$ is log smooth and the corresponding condition is satisfied.
		\item 
		If $(X, D)$ is almost $F$-liftable and $D=\bigcup D_i$, where $D_i\subset X$ are prime divisors then $(D_i, (\bigcup _{j\ne i }D_j)\cap D_i)$  is also almost $F$-liftable. This observation follows from the corresponding fact for $F$-liftable log smooth pairs (see \cite[Lemma 3.2]{AWZ2} for a simple proof).
	\end{enumerate}
\end{Remark}

We also need to introduce some notions of singularities in presence of liftings:

\begin{Definition}\label{definition-F-liftable-2}
	\begin{enumerate}
		\item If $(X,D)$ is liftable to $W_2(k)$ then we say that it is  \emph{locally $F$-liftable} if there exists a lifting $(\tilde X, \tilde D)$ of $(X, D)$ such that every $x\in X$ has an open neighbourhood $V\subset X$ for which there exists an $F$-lifting of $(V, D_V)$	compatible with the lifting induced from $(\tilde X, \tilde D)$.
		\item 
		If $(X,D)$ is almost liftable to $W_2(k)$ then we say that it is  \emph{locally $F$-liftable} if there exists a big open subset $U\subset X $ and a  lifting $(\tilde U, \tilde D_U)$ of a  $(U, D_U)$ such that 
		every point of $x\in X$ has an open neighbourhood $V\subset X$ for which there exists an $F$-lifting of $(V, D_V)$ compatible with the lifting of $(V\cap U, D_{V\cap U})$ induced from $(\tilde U, \tilde D_U)$.
		\item If $(X,D)$ is almost liftable to $W_2(k)$ then we say that $(X,D)$ \emph{has $F$-liftable singularities in codimension $2$} if there exists a closed subset $Z\subset X$ of codimension $\ge 3$ such that $(X\backslash Z, D\backslash Z)$ is  locally $F$-liftable.
	\end{enumerate}
\end{Definition}

\begin{Remark}\label{remarks-on-definition-loc-F-lift}
	\begin{enumerate}
		\item If $(X,D)$ is log smooth and (almost) liftable to $W_2(k)$  then it is also locally $F$-liftable.
		\item Note that	if there exists a big open subset $U\subset X$ such that the pair $(U, D_U=D\cap U)$ is liftable to $W_2(k)$ and locally $F$-liftable then $(X,D)$ is almost liftable to $W_2(k)$ but it does not need to be locally $F$-liftable.
	\end{enumerate}
\end{Remark}

\medskip

A characteristic $p$ scheme $X$ is called \emph{$F$-split} if there exists an $\cO_X$-linear map $\varphi :(F_X)_*\cO_X\to \cO_X$ splitting   $F_X^{\sharp}: \cO_X\to (F_X)_*\cO_X$.
If $Y_1,...,Y_s$ are closed subschemes of $X$ then we say that they are \emph{compatibly $F$-split} by $\varphi$ if $\varphi ((F_{X})_* I_{Y_j})\subset I_{Y_j}$ for all $j$.
For the basic facts about these notions we refer the reader to \cite{BK}.
In the proof of the next proposition we need the following generalization of the second part of \cite[Proposition 1.3.11]{BK}.

\begin{Lemma}\label{generalized-BK}
	Let $X$ be a smooth variety defined over an algebraically closed field of characteristic $p>0$.
	Let	us assume that $\varphi \in H^0(X, \omega_X^{1-p})\simeq \Hom _{\cO_X} ((F_X)_*\cO_X, \cO_X)$ splits $X$. Let $Z(\varphi)=(p-1)D+D'$ be the divisor of zeroes of $\varphi$, where $D$ and $D'$
	are effective divisors. Then $D$ is reduced and $\varphi$ splits $X$ compatibly with all irreducible components of $D$. 
\end{Lemma}

\begin{proof}
	Let $Y$ be an irreducible component of $D$ and let $x$ be a smooth point of the support of $Z(\varphi)$ that belongs to $Y$. Then we can choose a system of local coordinates $(t_1,...,t_n)$ at $x$ such that the local equation of $Y$ is given by $t_1=0$.
	Note that by assumption the local expansion of $\varphi$ at $x$ is given by
	$$t_1^{m(p-1)}g(t_1,...,t_n)(dt_1\wedge ...\wedge dt_n)^{1-p},$$ 
	where $g(t_1,...,t_n)$ is not divisible by $t_1$ and $m\ge 1$ is the multiplicity of $Y$ in $D$.	
	Since $\varphi $ splits $X$  by \cite[Theorem 1.3.8]{BK} the coefficient of the monomial $(t_1...t_n)^{p-1}$ in $t_1^{m(p-1)}g(t_1,...,t_n)$ is nonzero. Hence $m=1$  and the splitting $\varphi$ is compatible with $Y$ at $x$. It follows that $\varphi$ is compatible with $Y$ at smooth points of the support of $ Z (\varphi)$. So the required assertion follows from \cite[Lemma 1.1.7, (ii)]{BK}. 
\end{proof}

\medskip

\begin{Proposition}\label{liftable-split}
	Let $(X,D)$ be a log pair.
	\begin{enumerate}
		\item 
		If $X$ is $F$-split  compatibly with all irreducible components of $D$ then  $(X,D)$ is liftable to $W_2(k)$.		
		\item  
		If $(X,D)$ is almost $F$-liftable then $X$ is $F$-split compatibly with all irreducible components of $D$.
	\end{enumerate}
\end{Proposition}

\begin{proof}
	In case $D=0$ the first part is contained in \cite[Proposition 4]{La3}
	and the second one  follows from \cite[Theorem 2]{BTLM} (see also \cite[Section 2]{AWZ2}).
	In general, the first part follows from \cite[Lemma 5.2.1]{AZ}.	 
	 By \cite[Lemma 1.1.7, (ii) and (iii)]{BK} to prove the second part it is sufficient to prove that if $(X,D)$ is log smooth and $F$-liftable then irreducible components of $D$ are compatibly $F$-split. 
	Note that the $F$-splitting induced by a lifting $\tilde F_X$ that is compatible with $\tilde D$
	vanishes to order $(p-1)$ along $D$ (see the proof of \cite[Lemma 3.1]{AWZ2}). So we can conclude by Lemma \ref{generalized-BK}.
\end{proof}

\begin{Remark}
	If $(X,D)$ is log smooth then the fact that $D$ is compatibly $F$-split is claimed in
	\cite[Lemma 3.1]{AWZ2} but the proof there contains a gap. 
	The problem is that the Frobenius splitting coming from the lifting of the Frobenius morphism to $W_2(k)$ does not need to come from $(p-1)$-th power of a section of $H^0(X, \omega_X^{-1})$. 
See below for an explicit example.	
\end{Remark}

\begin{Example}
Let us consider divisor $D:=(x_1=0)\subset X:=\Spec k[x_1,x_2]$, where $k$ is a perfect field of characteristic $p>2$. Let $\tilde X:=\Spec W_2(k)[x_1,x_2]$ be a lifting of $X$ to $W_2(k)$
and let $\tilde D:=(x_1=0)\subset \tilde X$ be a lifting of $D\subset X$.
Let us consider a lifting $\tilde F_X$ of $F_X$ given by $x_1\to x_1^p$ and $x_2\to x_2^p+px_2^2$. This lifting is compatible with $\tilde D$. However, it is easy to see that the Frobenius splitting associated to $\tilde F_X$ is given by 
 $$\varphi=x_1^{p-1}x_2(x_2^{p-2}+2)(dx_1\wedge dx_2)^{1-p}\in H^0(X, \omega _X^{1-p}),$$
so $\varphi$ is not  a $(p-1)$-th power of a section of $H^0(X, \omega_X^{-1})$. In fact, in this case one cannot find any open subset $U\subset X$ such that  $\varphi|_U$ is a $(p-1)$-th power of a section of $H^0(U, \omega_U^{-1})$. On $U=\{x_2(x_2^{p-2}+2)\ne 0\}$ one can multiply $\varphi$ by an invertible $u\in \Gamma (U, \cO_U^*)$ so that $u\cdot \varphi|_U=\psi ^{p-1}$ for some $\psi \in H^0(U, \omega_U^{-1})$ and apply \cite[Proposition 1.3.11]{BK} to this new splitting.
This shows that $u\cdot \varphi|_U$ splits $U$ compatibly with $D\cap U$. However, this is not sufficient to apply \cite[Lemma 1.1.7, (ii)]{BK} to conclude that $\varphi $ splits $X$ compatibly with $D$.
\end{Example}

\medskip

\begin{Example}
	The following example is motivated by \cite[Example 5.1]{Zd} (note that the argument showing $F$-liftability works for $p>2$;  for $p=2$ $F$-liftability needs to be proven using \cite[Corollary 4.12]{Zd}).
	
	Let us consider divisor $D:=(x_1x_2(x_1+x_2)=0)\subset X:=\Spec k[x_1,x_2]$, where $k$ is a perfect field of characteristic $p>0$. Then $\tilde X:=\Spec W_2(k)[x_1,x_2]$ is a lifting of $X$ to $W_2(k)$ and it has a natural  lifting $\tilde F_X$ of $F_X$
	given by $x_i\to x_i^p$ for $i=1,2$. Let $\tilde D:=(x_1x_2(x_1+x_2)=0)\subset \tilde X$ be a lifting of $D\subset X$. If $p>2$ then $\tilde F _X$ induces a compatible lifting $\tilde F_X|_{\tilde D}: \tilde D\to \tilde D$ of $F_D$. However,  $\tilde F_X$ {is not compatible with} $\tilde D$ as $\tilde F_X^*\tilde D=(x_1^px_2^p(x_1^p+x_2^p)=0)$ is not equal to $p\tilde D=(x_1^px_2^p(x_1+x_2)^p=0)$.
	In fact, an explicit computation shows that $(X, D)$ is not $F$-liftable. Note however that there exist splittings of $X$ that are compatible with $D$. For example, one can take splitting of $X$ corresponding to 
	$$\varphi=x_1^{p-1}x_2^{p-1}(x_1+x_2)^{p-1}(dx_1\wedge dx_2)^{1-p}\in H^0(X, \omega _X^{1-p}).$$
\end{Example}

\medskip

We need also the following logarithmic version of  \cite[Theorem 3.3.6 (a), (iii)]{AWZ}. The proof is analogous to the one from \cite{AWZ} and we leave it to the reader.

\begin{Lemma}\label{AWZ-lemma}
Let $(X,D)$ be a log scheme and let $U\subset X$ be a big open subset of $X$. Let $(\tilde X,\tilde D)$ be a $W_2(k)$-lifting of $(X,D)$ and let $F_{\tilde U}$ be an $F$-lifting of  $(\tilde U,\tilde D_U)$, where $\tilde U=(U, \cO_{\tilde X}|_U)$ and $\tilde D_U=(D_U, \cO_{\tilde D}|_{D_U})$. Then there exists 
an  $F$-lifting $\tilde F_X: \tilde X\to \tilde X$ compatible with $\tilde D$.
\end{Lemma}

The following theorem shows that an almost liftable log pair, which is locally almost $F$-liftable is already liftable to $W_2$ and locally $F$-liftable.

\begin{Theorem}\label{equivalences-for-almost-liftings}
Let $(X,D)$ be a log pair. Then the following conditions are equivalent:
\begin{enumerate}
	\item $(X,D)$ is liftable to $W_2(k)$ and it is locally $F$-liftable.	
	\item $(X,D)$ is almost liftable to $W_2(k)$ and it is locally $F$-liftable.
	\item There exists a big open subset $U\subset X$  and a lifting  $(\tilde U, \tilde D_U)$ of $(U, D_U=D\cap U)$ such that every $x\in X$ has an open neighbourhood $V\subset X$ for which there exists an almost $F$-lifting of $(V, D_V)$ compatible with the almost lifting  induced from $(\tilde U, \tilde D_U)$. 
\end{enumerate}
\end{Theorem}

\begin{proof}
Implications $(1)\Rightarrow (2)$ and $(2)\Rightarrow (3)$ are clear. So let us assume (3). 
Then every point $x\in X$ has an open neighbourhood $V\subset X$ and a big open subset $V'\subset V\cap U$ with a compatible $F$-lifting $F_{\tilde V'}:(\tilde V', \tilde D') \to (\tilde V', \tilde D')$, where $\tilde V'=(V', \cO_{\tilde U}|_{V'})$ and  $\tilde D'=(D', \cO_{\tilde D}|_{D'})$.
By Lemma \ref{AWZ-lemma} we can extend $F_{\tilde V'}$ to  $F_{\widetilde{V\cap U}}$, where $\widetilde{V\cap U}:= (V\cap U, \cO_{\tilde U}|_{V\cap U})$. Moreover,  $F_{\widetilde{V\cap U}}$ is compatible with $\tilde D _{U\cap V}:= (V\cap D_U, \cO_{\tilde D_U}|_{V\cap D_U})$.

By Proposition \ref{liftable-split} we know that $V$ is $F$-split compatibly with irreducible components of $D_V$ and hence we have a canonical lifting of 
$(V,D_V)$ to $W_2$. Moreover, this lifting extends lifting $(\widetilde{V\cap U}, \tilde D _{U\cap V})$. So again using  Lemma \ref{AWZ-lemma} we can extend  $F_{\widetilde{V\cap U}}$ to an $F$-lifting of 
$(V, D_V)$. This shows (2). 

Now let us remark that for all $x$ we can glue the obtained canonical liftings $(\tilde V, \tilde D_V)$ to $(\tilde U, \tilde D_U)$, obtaining a lifting of  $(X,D)$ to $W_2(k)$, which is locally $F$-liftable.
One can do that since an $F$-lifting is uniquely determined up to a canonical isomorphism (this is a log version of \cite[Theorem 2.7]{AWZ2}). 
\end{proof}

The above theorem immediately implies the following corollary:

\begin{Corollary}
If $(X,D)$ is almost liftable to $W_2(k)$ and it has $F$-liftable singularities in codimension $2$
then there exists a closed subset $Z\subset X$ of codimension $\ge 3$ such that $(X\backslash Z, D\backslash Z)$ is liftable to $W_2(k)$ and it is locally $F$-liftable.
\end{Corollary}

\begin{Remark}
 Note that it is usually much easier to lift to $W_2(k)$ a big open subset of $X$ than the whole $X$.
	For example, if $X$ is a smooth projective surface then any open subset $U\subsetneq X$ is liftable to $W_2(k)$. This follows from the fact that the obstruction  to lifting of $U$ to $W_2(k)$ lies in $H^2 (T_U)$, which vanishes by Lichtenbaum's theorem (see \cite[Theorem 0G5F]{SP}). Theorem 
\ref{equivalences-for-almost-liftings} says that if $X$ is not liftable to $W_2(k)$ then it is not locally (almost) $F$-liftable with respect to any lifting of $U$.
\end{Remark}

\begin{Remark}
If $X$ is $F$-liftable then it does not need to have rational singularities.
In fact, \cite[Example 5.2]{Zd} shows that the cone over an ordinary elliptic curve is $F$-liftable. This singularity is log canonical but not klt. Let us also recall that by \cite[Theorem 2.10, (c)]{AWZ2} if $X$ is $F$-liftable and $G$ is a linearly reductive group acting on $X$ then the quotient $X\sslash G$ is also $F$-liftable.

Finally, note that by \cite[Theorem 4.15]{Zd} ordinary double points of the form $(x_1^2+...+x_n^2=0)\subset \AA^n_k$ for $n\ge 5$ in characteristic $p\ge 3$  are $F$-split but they are not (locally) $F$-liftable.
 These singularities  are not only log canonical but even  terminal.  The hypersurface $(x_1^2+...+x_n^2=0)\subset \AA^n_k$ is almost liftable to $W_2(k)$  and it has $F$-liftable singularities in codimension $2$ (since for $n\ge 4$ it is regular in codimension $2$).
\end{Remark}

\subsection{Intersection theory on normal varieties}\label{Subsection:intersection-theory}

Let $X$ be a normal projective variety of dimension $n$ defined  over an algebraically closed field $k$.
In the following we write $A^1(X)$ for the class group of $X$, i.e., the group of Weil divisors modulo 
rational equivalence on $X$. If $\cE$ is a coherent $\cO_X$-module of rank $r\ge 1$ then the sheaf $\det \cE =(\bigwedge ^r \cE)^{**}$ is  reflexive of rank $1$ and we can consider the associated class $c_1 (\cE)\in A^1(X)$ of Weil divisors on $X$. 

More generally, we write $A_m (X)$ for the group of $m$-cycles modulo rational equivalence on $X$. Chern classes of vector bundles on $X$ are considered as in \cite{Fu} as operations on $A_*(X)$.

We say that two line bundles $L$ and $M$ on $X$ are \emph{numerically equivalent}
if for every proper curve $ C\subset X$ we have
$$\int_Xc_1(L)\cap [C]= \int_Xc_1(M)\cap [C].$$
If $L$ is numerically equivalent to $\cO_X$ then we say that $L$ is \emph{numerically trivial}. The group of line bundles modulo numerical equivalence is denoted by 
$N^1(X)$. This is a torsion free quotient of the N\'eron--Severi group of $X$. So by theorem of the base, $N^1(X)$ is a free $\ZZ$-module of finite rank. 

\medskip

Below we recall some results from \cite{La-Chern}.

\begin{Theorem}\label{main1}
	For any Weil divisors $D_1$ and $D_2$ on $X$ there exists a $\ZZ$-multilinear symmetric form $N^1(X)^{\times (n-2)}\to \QQ$, $(L_1,... ,L_{n-2})\to D_1.D_2.L_1...L_{n-2}$ such that:
	\begin{enumerate}	
		\item If both $D_1$ anf $D_2$ are Cartier then  
		$$ D_1.D_2.L_1...L_{n-2}= \int_Xc_1(\cO_X (D_1))\cap c_1(\cO_X(D_2))\cap c_1(L_1)\cap ...\cap c_1(L_{n-1})\cap  [X].$$
		\item If $D_2$ is a Cartier divisor then  
		$$ D_1.D_2.L_1...L_{n-2}= \int_X c_1(\cO_X (D_2))\cap c_1(L_1)\cap ...\cap c_1(L_{n-2})\cap  [D_1]\in \ZZ.$$
		\item If $L_1,...,L_{n-2}$ are very ample
		then   for a general complete intersection surface  $S\in |L_1|\cap ...\cap |L_{n-2}|$ we have 
		$$D_1.D_2.L_1...L_{n-2} =D_1|_S.D_2|_S,$$
		where on the right hand side  we have Mumford's intersection of Weil divisors on a normal surface.
	\end{enumerate}
\end{Theorem}

\begin{Theorem}\label{properties-of-ch_2}
	Assume that $k$ has positive characteristic. For any normal projective variety $X/k$ and for any 
	coherent reflexive $\cO_X$-module $\cE$ on $X$ there exists a $\ZZ$-multilinear symmetric 
	form $\int _X \ch _2 (\cE): N^1(X) ^{\times (n-2)}\to \RR$ such that:
	\begin{enumerate}
		\item If $\cE$ is a vector bundle on $X$ then $$\int _X \ch _2 (\cE)L_1...L_{n-2}=\int_X \ch _2 (\cE)\cap c_1(L_1)\cap ...\cap c_1(L_{n-2})\cap[X].$$
	    \item If $k\subset K$ is an algebraically closed field extension then 
	    $$\int _{X_K} \ch _2 (\cE_K)(L_1)_K...(L_{n-2})_K=\int _X \ch _2 (\cE)L_1...L_{n-2}.$$
	    \item If $n>2$ and  $ L_1$ is very ample then for a very general hypersurface $H\in |L_1|$ we have
		$$\int _X \ch_2 (\cE)L_1...L_{n-2}=\int _{H}\ch_2 (\cE|_{H})L_2|_{H}...L_{n-2}|_{H}.$$
		\item If $X$ is a surface then
		$$\int _X \ch _2 (\cE)= \lim _{m\to \infty }\frac {  \chi (X,  F_X^{[m]}\cE  ) }{p^{2m}}.$$
		\item We have
		$$\int _X \ch _2 (F_X^{[*]}\cE)L_1...L_{n-2}=p^2\int _X \ch _2 (\cE)L_1...L_{n-2}.$$
	\end{enumerate}
\end{Theorem}

Once we have the above theorems we can define some other Chern numbers as follows.

\begin{Definition}
	For any reflexive coherent $\cO_X$-module $\cE$ of rank $r$ and any line bundles $L_1,...,L_{n-2}$ on $X$
	we define the following Chern numbers:
	$$\int_X c_1 ^2(\cE) L_1...L_{n-2}:= c_1 (\cE)^2.L_1...L_{n-2},$$
	$$\int _X c_2 (\cE)L_1...L_{n-2}:= \frac{1}{2} \int_X c_1 ^2(\cE) L_1...L_{n-2}-\int _X \ch _2 (\cE)L_1...L_{n-2},$$
	$$\int _X \Delta (\cE)L_1...L_{n-2}=2r \int _X c_2 (\cE)L_1...L_{n-2} -(r-1)\int_X c_1 ^2(\cE) L_1...L_{n-2}.$$
\end{Definition}

By linearity we can also extend obtained forms to $\QQ$-line bundles. In this way we get symmetric $\QQ$-multilinear forms $N^1(X)_{\QQ}^{n-2}\to \RR$, where $N^1(X)_{\QQ}=N^1(X)\otimes \QQ$.

\subsection{Numerical groups of divisors}

Let $X$ be an irreducible normal scheme defined over an algebraically closed field $k$.
We say that  a Weil divisor $D$ is \emph{algebraically equivalent to zero} if there exists a smooth variety $T$, a Weil divisor $ G$ on $X\times _kT$ and $k$-points $t_1,t_2\in T$ 
such that $D= G_{t_1}-G_{t_2}$ in $A^1(X)$  (see \cite[10.3]{Fu}). Then we write $D\sim _{alg}0$. 
The group of algebraic equivalence classes of Weil divisors on $X$ is denoted by  $B^1(X)$. 
By $\bar B^1(X)$ we denote the quotient of $B^1(X)$ by torsion.

Let us recall that if $X$ is proper then $N^1(X)\simeq \Pic X/\Pic ^{\tau} X$.
If $X$ is also smooth  then $N^1(X)\simeq \bar B^1(X)$ (see \cite[Theorem 9.6.3]{Kl}).  
We will need the following variant of theorem of the base. It is a special case of \cite[Th\'eor\`eme 3]{Ka} but we provide a different simple proof in the case used in the paper.

\begin{Lemma}\label{Neron-Severi}
	Let $X$ be an irreducible, normal, proper scheme defined over an algebraically closed field $k$. Then $\bar B^1(X)$ is a free $\ZZ$-module of finite rank.
\end{Lemma}

\begin{proof} 
By definition $\bar B^1(X)$ is torsion free. So it is sufficient to prove that $\bar B^1(X)$ is  finitely generated as a $\ZZ$-module. By \cite{dJ} there exists an alteration $\tilde Y\to X$ from  a smooth projective variety $\tilde Y$. Taking Stein's factorization we get a proper birational map
	$g: \tilde Y\to Y$ to normal variety and a finite surjective morphism $\pi: Y\to X$.
	
	Let $E$ be the exceptional locus of $g$. Then using the localization sequence (see \cite[Example 10.3.4]{Fu}) we get a surjective map
	$$ B^1 (\tilde Y )\twoheadrightarrow B^1 (\tilde Y\backslash E)\simeq B^1 (Y\backslash g(E))
	\simeq B^1(Y),$$
showing  that $\bar B^1(Y)$ is finitely generated.
	
There exists a big open subset $U\subset X$ such that $\pi: V:=\pi^{-1}(U)\to U$ is flat.
Then using flat pullback and the localization sequence (see \cite[Proposition 10.3 and Example 10.3.4]{Fu}) we have a well defined map
$$B^1(X)\simeq B^1(U)\mathop{\longrightarrow}^{\pi^*} B^1(V)\simeq B^1(Y)$$
induced by pullback of Weil divisors. Since  $\pi_*\pi^*$ is multiplication by the degree of $\pi$ on $A^1(U)$ (and hence also on $B^1(U)$), the induced map $\bar B^1(X)\to \bar B^1(Y)$ is injective.
This implies that  $\bar B^1(X)$ is also finitely generated.
\end{proof}

\medskip

From now one in this subsection $X$ is a normal projective variety of dimension $n$ defined  over an algebraically closed field $k$. 

\begin{Lemma}
If a Weil divisor $D_1$ is algebraically equivalent to zero then for every Weil divisor $D_2$
and all line bundles $L_1,..., L_{n-2}$ we have $$D_1.D_2.L_1...L_{n-2}=0.$$
\end{Lemma} 

\begin{proof}
Let us first assume that $X$ is a surface and let $f: \tilde X\to X$ be a resolution of singularities. By assumption there exists a smooth variety $T$, a Weil divisor $G$ on $X\times _kT$ and $k$-points $t_1,t_2\in T$ such that $D= G_{t_1}-G_{t_2}$ in $A^1(X)$. 
Let us consider the map $g:=f\times \id : \tilde X\times _kT\to X \times_k T$. One can use Mumford's construction of pullback to define  $g^*G$ that restricts to $f^*(G_t)$ on $\tilde X\times \{t\}$
for every $t\in T(k)$.
Then we have $f^*D= (g^*G)_{t_1}-(g^*G)_{t_2}$ in $A^1(\tilde X)\otimes \QQ$. This implies  that some multiple of $f^* D_1$ is algebraically equivalent to zero and hence $f^*D_1.f^*D_2=D_1.D_2=0$.

	In general, by linearity of the intersection product it is sufficient to prove that  $D_1.D_2.L_1...L_{n-2}=0$ assuming that $L_1,..., L_{n-2}$ are very ample.
Let $S\in |L_1|\cap ...\cap |L_{n-2}|$ be a general complete intersection surface.
Since cycles algebraically equivalent to zero are preserved by  Gysin homomorphisms (see \cite[Proposition 10.3]{Fu}) the restriction $D_1|_S$ is algebraically equivalent to zero.
So by Theorem \ref{main1} we have
$$D_1.D_2.L_1...L_{n-2} =D_1|_S.D_2|_S=0.$$
\end{proof}

The above lemma shows that the intersection pairing $(D_1, D_2, L_1,... ,L_{n-2})\to D_1.D_2.L_1...L_{n-2}$ induces a $\ZZ$-multilinear map 
$$B^1(X)\times B^1(X)\times N^1(X)^{\times (n-2)}\to \QQ .$$

\medskip

Let us fix  a collection $L=(L_2,...,L_{n-1})$ of nef line bundles on $X$.
Assume that there exists a nef line bundle $L_1$ such that $L_1L_2....L_{n-1}$ is numerically non-trivial, i.e., there exists some Weil divisor $D$ such that $D.L_1...L_{n-1}\ne 0$. Let us consider a $\QQ$-valued intersection pairing $\langle \cdot, \cdot \rangle _L : B^1(X)\times B^1(X) \to \QQ$ defined by
$$\langle D_1, D_2\rangle _L:= D_1.D_2.L_2...L_{n-1} .$$
Let us write $N_L(X)$ for the quotient of $B^1 (X)$ modulo the radical of this intersection pairing. 
Then we have an induced non-degenerate intersection pairing
$$\langle \cdot, \cdot \rangle _L: N_L(X)\times N_L(X)\to \QQ.$$

\begin{Lemma}\label{easy-HIT}
	Let us assume that $L_1^2L_2...L_{n-1}>0$. If  $D_1.L_1L_2...L_{n-1}=0$ then  $D_1^2L_2...L_{n-1}\le 0$
	with equality if and only if the class $[D_1]\in N_L(X)$ is zero.
\end{Lemma}

\begin{proof}
	Let us fix some ample line bundle $A$. Then the $\QQ$-line bundles $L_i+\epsilon A$ are ample for $\epsilon\in \QQ_{>0}$. So by \cite[Lemma 2.6]{La-Chern} we have inequalities
	$$D_1^2(L_2+\epsilon A) ...(L_{n-1}+\epsilon A)\cdot  (L_1+\epsilon A)^2(L_2+\epsilon A) ...(L_{n-1}+\epsilon A) \le \left( D_1.(L_1+\epsilon A) ...(L_{n-1}+\epsilon A) \right)^2.$$
	Taking the limit when $\epsilon\to 0$, we get  $D_1^2L_2...L_{n-1}\le 0$. Now let us assume that 	
	$D_1^2L_2...L_{n-1}= 0$ but
	$D_1.D_2.L_2...L_{n-1}\ne 0$ for some Weil divisor $D_2$. Replacing $D_2$ by $(L_1^2L_2...L_{n-1})D_2-(D_2.L_1...L_{n-1})L_1$
	we can assume that $D_2.L_1...L_{n-1}=0$.
	Therefore we have
	$$0\ge (tD_1+D_2)^2.L_2...L_{n-1}=2t D_1.D_2.L_2...L_{n-1}+D_2^2.L_2...L_{n-1},$$
	which gives a contradiction with some $t\in \ZZ$.	
\end{proof}

\medskip

The following lemma generalizes  \cite[Lemma 2.6]{La-Chern}.

\begin{Lemma}\label{HIT}
$N_L (X)$ is a free $\ZZ$-module of finite rank. If  $\rk _{\ZZ} N_L (X)=s$ then the intersection pairing $\langle \cdot , \cdot \rangle _L$ has signature $(1, s-1)$. 
\end{Lemma}
 
\begin{proof}
By definition $N_L(X)$ is torsion free and it is a quotient of $B^1 (X)$. So by Lemma \ref{Neron-Severi} $N_L(X)$  is also a free $\ZZ$-module of finite rank. 
If $L_1^2L_2...L_{n-1}>0$ then  the second assertion follows from Lemma \ref{easy-HIT}.
In general,  there exists some Weil divisor $D$ such that $D.L_1...L_{n-1}\ne 0$. Without loss of generality we can assume that 
$D.L_1...L_{n-1}>0$. Then for every ample Cartier divisor $H$ we have $H^0(X,\cO_X(mH-D))\ne 0$ 
for $m\gg 0$. So  $mHL_1...L_{n-1}\ge D.L_1...L_{n-1}>0$. Then $M=(L_1+H)$ is ample and $M^2L_2....L_{n-1}>0$.  Since  the definition of $N_L(X)$ does not depend on $L_1$, we get the required assertion from the previous case.
\end{proof}
 
As in \cite[Lemma 2.1]{La7} the above lemma implies the following result (in fact, the first part follows from the proof of Lemma \ref{HIT}). 
 
\begin{Corollary}\label{HIT2}
If $H$ is an ample line bundle then $HL_1L_2...L_{n-1}>0$. Moreover, if
$D.L_1L_2...L_{n-1}=0$ for some Weil divisor $D$ then
then  $D^2.L_2...L_{n-1}\le 0$.
\end{Corollary}

\medskip

\section{Several auxiliary results on Chern classes}

In this section we prove several results on Chern classes of reflexive sheaves that will be needed throughout the paper.
We assume that $X$ is a normal projective variety of dimension $n$ defined over an algebraically closed field $k$ of characteristic $p>0$. We also fix a collection
$(L_1,...,L_{n-1})$ of nef line bundles on $X$. 

\begin{Lemma}\label{short-exact-seq}
Let $$0\to \cE_1\to \cE \to \cE_2\to 0$$
be a left exact sequence of reflexive sheaves on $X$, which is also right exact 
on some big open subset of $X$. Then we have
$$\int _X \ch _2 (\cE)L_1...L_{n-2} \le \int _X \ch _2 (\cE _1) L_1...L_{n-2}  +\int _X \ch _2 (\cE _2)L_1...L_{n-2} .$$
Moreover, if the above sequence is exact on $X$ and  locally split in codimension $2$ then we have equality
$$\int _X \ch _2 (\cE)L_1...L_{n-2} = \int _X \ch _2 (\cE _1)L_1...L_{n-2} +\int _X \ch _2 (\cE _2)L_1...L_{n-2} .$$
\end{Lemma}

\begin{proof}
Since numerical equivalence classes of nef line bundles are limits of classes of ample $\QQ$-line bundles, we can by continuity assume that all $L_i$ are ample $\QQ$-line bundles. Passing to their multiples we can also assume that all $L_i$ are very ample line bundles. By Theorem \ref{properties-of-ch_2}, (2) and (3), we can assume that the base field $k$ is uncountable and then
by restricting to a very general complete intersection surface $S\in |L_1|\cap ...\cap |L_{n-2}|$ we can assume that $X$ is a surface.

Let $U$ be a big open subset on which all $\cE$, $\cE_1$ and $\cE_2$ are locally free and let $j: U\hookrightarrow X$
denote the open embedding. Since $X$ is normal, we can also assume that $U$ is contained in the regular locus $X_{reg}$ of $X$. Since $F_U^*$ is exact, the sequences 
$$0\to (F_U^m)^*\cE_1\to (F_U^m)^* \cE \to (F_U^m)^*\cE_2\to 0$$
are exact. This implies that sequences
$$0\to F_X^{[m]}\cE_1\to F_X^{[m]} \cE \to F_X^{[m]} \cE_2 $$
are exact and the cokernel of the last map is supported on the closed subset $X\backslash U$ of codimension $\ge 2$. So we get inequalities 
$$ \chi (X, F_X^{[m]} \cE ) )\le  \chi (X, F_X^{[m]} \cE_1 )+ \chi (X, F_X^{[m]} \cE _2) .$$
Dividing by $p^{2m}$, passing to the limit and using Theorem \ref{properties-of-ch_2}, (3), we get the required inequality.
Equality follows from the fact that
the above mentioned left exact sequence becomes right exact if the sequence 
 $$0\to \cE_1\to \cE \to \cE_2\to 0$$
is locally split.
\end{proof}

\medskip

\begin{Example}\label{singular-quadric}
If the short exact sequence in the above lemma is not locally split, then it is well known that the  inequality can be strict. For example, if $X\subset \PP ^3$ is the cone over a smooth quadric curve in $\PP ^2$ and $D$ is its generator, then we have a short exact sequence
$$0\to \cO_X (-D)\to  \cO_X\oplus \cO_X \to \cO_X (D)\to 0$$
with
$$\int _X \ch _2 ( \cO_X\oplus \cO_X)=0<\int _X \ch _2 ( \cO_X (-D))+\int _X \ch _2 ( \cO_X (D))= D^2=\frac{1}{2}.$$
This example shows also that Chern classes of reflexive sheaves on a normal projective surface
are not deformation invariant, i.e., they can change in flat families.
\end{Example}

\medskip

\begin{Remark}
Equality in Lemma \ref{short-exact-seq} is one of the results that is  not known for general normal surfaces defined over an algebraically closed field of characteristic $0$ (see \cite{La-Chern}). 
\end{Remark}

\medskip

\begin{Lemma}\label{Chern-classes-filtrations}
Let $\cE$ be a reflexive coherent $\cO_X$-module and let $\cE=N^0\supset
N^1\supset ...\supset N^s=0$ be a filtration such that the associated graded $\Gr
_N(\cE)=\bigoplus _i N^i/N^{i+1}$ is torsion free. Let $\cF$ be the reflexive hull of  $\Gr
_N(\cE)$. Then the following conditions are satisfied:
\begin{enumerate}
\item $c _1 (\cE )=c_1 (\cF)$,
\item $\int _X \ch _2 (\cE)L_1...L_{n-2} \le \int _X \ch _2 (\cF)L_1...L_{n-2} $,
\item $\int _X \Delta  (\cE)L_1...L_{n-2}\ge \int _X \Delta  (\cF)L_1...L_{n-2} $.
\end{enumerate}
\end{Lemma}

\begin{proof}
The first condition is clear as  $c _1 (\cE )=c_1(\Gr_N(\cE))=c_1 (\cF)$.
We prove the second condition by induction on the length $s$ of the filtration.
If $s>1$ then $N^1$ is reflexive as $N^0/N^1$ is torsion free. So by Lemma \ref{short-exact-seq}
we have
$$\int _X \ch _2 (\cE)L_1...L_{n-2} \le \int _X \ch _2 (N^1)L_1...L_{n-2}+  \int _X \ch _2 ((N^0/N^1)^{**})L_1...L_{n-2}.$$
Applying the induction assumption to the filtration $N^1\supset N^2\supset ...\supset N^s=0$ of $N^1$,
we get the required inequality. The last condition follows from (1) and (2).
\end{proof}

\begin{Lemma}\label{needed-for-boundedness}
Let $\cE$ be a rank $r$ reflexive coherent $\cO_X$-module and let $\cE=N^0\supset
N^1\supset ...\supset N^s=0$ be a filtration such that all $N^i/N^{i+1}$ are torsion free. 
Let us assume that $L$ is numerically nontrivial and let us
set $\cF_i:=(N^i/N^{i+1})^{**}$, $r_i:=\rk \cF_i$ and $\mu_i:= \mu _L(\cF_i)$. 
\begin{enumerate}
	\item 
	If $d:=L_1^2L_2....L_{n-1}>0$ then
	\begin{align*}
		\frac{\int _X \Delta  (\cE)L_2...L_{n-1}}{r}\ge \sum _i \frac{\int _X \Delta  (\cF_i)L_2...L_{n-1}}{r_i}-
		\frac{1}{rd}\sum _{i<j} r_ir_j\left(\mu_i -  \mu_j\right)^2.
	\end{align*}
	\item 	If $\mu_i=\mu_L(\cE)$ for all $i$ then
	\begin{align*}
		\frac{\int _X \Delta  (\cE)L_2...L_{n-1}}{r}\ge \sum _i \frac{\int _X \Delta  (\cF_i)L_2...L_{n-1}}{r_i}.
	\end{align*}
\end{enumerate}
\end{Lemma}

\begin{proof}
Passing to an algebraically closed and uncountable field extension of $k$ we can assume that $k$ is uncountable. If we set $\cF:=\bigoplus \cF_i$ then Lemma \ref{short-exact-seq} gives
$$\int _X \ch _2 (\cF)L_1...L_{n-2} = \sum _i \int _X \ch _2 (\cF _i)L_1...L_{n-2} .$$
After rewriting this gives
$$
\frac{\int _X \Delta  (\cF)L_2...L_{n-1}}{r}= \sum _i \frac{\int _X \Delta  (\cF_i)L_2...L_{n-1}}{r_i}-
\frac{1}{r}\sum _{i<j} r_ir_j\left( \frac{c_1\cF _i}{r_i} -  \frac{c_1\cF _j}{r_j} \right)^2.L_2...L_{n-1}.$$
But by the Hodge index theorem (see Lemma \ref{HIT}) we have
$$\left(\mu_i -  \mu_j\right)^2 =\left( \left( \frac{c_1\cF _i}{r_i} -  \frac{c_1\cF _j}{r_j} \right).L_1...L_{n-1}\right)^2
\ge d\cdot   \left( \frac{c_1\cF _i}{r_i} -  \frac{c_1\cF _j}{r_j} \right)^2.L_2...L_{n-1}.$$
So (1) follows from Lemma \ref{Chern-classes-filtrations}, (3). Under assumption (2), Corollary \ref{HIT2}
implies that 
$$\left( \frac{c_1\cF _i}{r_i} -  \frac{c_1\cF _j}{r_j} \right)^2.L_2...L_{n-1}\le 0,$$
so again the inequality follows from Lemma \ref{Chern-classes-filtrations}, (3).
\end{proof}

\begin{Lemma}\label{Delta-computation}
	Let  $H\in |L_1|$ be a normal variety and let $\cT$ be a rank $\tau$ torsion free $\cO_D$-module  and let $i: H\hookrightarrow X$ be the closed embedding. Let
	$$0\to \cG\to 
	\cE \to i_*\cT\to 0$$
	be a short exact sequence of coherent $\cO_X$-modules, where $\cE$ is reflexive. 
	Then 
	\begin{align*}
		\int_X \Delta (\cG)L_2\dots L_{n-1}
		=&\int_X \Delta (\cE)L_2\dots L_{n-1}-\tau (r-\tau ) L_1^2L_2\dots L_{n-1}\\
		&+2 (r c_1(\cT)-\tau c_1(i^*\cE)). i^*L_2\dots i^*L_{n-1}.
	\end{align*}
\end{Lemma}

\begin{proof}
	Note that $\cG$ is a coherent reflexive $\cO_X$-module.
	Since both sides of our inequality depend continuously on $L_2,...,L_{n-1}$ when considered as functions on $N^1(X)_{\QQ}$, and the inequality does not change when we pass to multiples,  we can assume that $L_2,..., L_{n-1}$ are very ample. By Theorems \ref{main1} and \ref{properties-of-ch_2} we can assume that the base field $k$ is uncountable and  then we can restrict to a very general complete intersection surface in $|L_2|\cap ...\cap |L_{n-1}|$ to reduce the assertion to the surface case. 
	An exact sequence
	$$(F_X^m)^*\cG\to 
	(F_X^m)^*	\cE \to (F_X^m)^* (i_*\cT)\to 0$$
	leads to 
	$$0\to F_X^{[m]}\cG\to 
	F_X^{[m]} \cE \to \cT_m\to 0, $$
	where $\cT_m $ is set-theoretically supported on $H$.
	Moreover, we have a canonical map $(F_X^m)^* (i_*\cT)\to \cT_m$, which is an isomorphism 
	on the set where $F_X$ is flat, i.e., on $X_{reg}$. 
	But $X$ is a surface and $H$ is a smooth curve, which is also a Cartier divisor. So $H$ is contained in $X_{reg}$ and hence $\cT_m\simeq (F_X^m)^* (i_*\cT)$. 
	This gives
	\begin{align*}
		\int _X \ch_2 (\cE)	&=\lim _{m\to \infty }\frac{\chi (X, F_X^{[m]} \cE)}{p^{2m}}=\lim _{m\to \infty }\frac{\chi (X,  F_X^{[m]} \cG)}{p^{2m}}+\lim _{m\to \infty }\frac{\chi (X, (F_X^m)^* (i_*\cT))}{p^{2m}}\\
		&=\int _X \ch_2 (\cG) +\lim _{m\to \infty }\frac{\chi (X, (F_X^m)^* (i_*\cT))}{p^{2m}}.
	\end{align*}
	To compute the last limit, let us consider a resolution of singularities $f: \tilde X\to X$, which is an isomorphism on $X_{reg}$. So we have a closed embedding $\tilde i: H\hookrightarrow \tilde X$ such that $f\circ \tilde i=i$.	 Then we have a short exact sequence
	$$0\to \tilde \cG\to  f^{[*]}\cE \to \tilde i_*\cT\to 0,$$
	where $\tilde \cG$ is a rank $r$ vector bundle. 
	Then by the same arguments as above we have
	\begin{align*}
		\lim _{m\to \infty }\frac{\chi (X, (F_X^m)^* (i_*\cT))}{p^{2m}}&=
		\lim _{m\to \infty }\frac{\chi (\tilde X, (F_{\tilde X}^m)^* (\tilde i_* \cT))}{p^{2m}}=
		\int _{\tilde X} \ch_2 (f^{[*]}\cE)- \int _{\tilde X} \ch_2 (\tilde \cG)\\
		& = \int _{\tilde X} \ch_2 (\tilde i_*\cT).
	\end{align*}
	To compute $\int _{\tilde X} \ch_2 (\tilde i_*\cT)$ one can use the Riemann--Roch theorem on $\tilde X$ and on $H$ to get
	$$ \deg _H \cT+\tau \chi (\cO_H)=\chi (H, \cT)=\chi (X,\tilde i_* \cT)= \int _{\tilde X} \ch_2 (\tilde i_*\tilde \cT)- \frac{1}{2}c_1(\tilde i_*\tilde \cT)K_{\tilde X}$$
	Since $c_1(\tilde i_* \cT)= \tau H$ and $ \chi (\cO_H)=-\frac {1}{2}H(K_{\tilde X}+H)$ (e.g, because $K_H=(K_{\tilde X}+H)|_H$ and $\deg K_H=-2 \chi (\cO_H)$), we get 
	$$ \int _{\tilde X} \ch_2 (\tilde i_*\cT)=\deg _H \cT -\frac {\tau}{2}H^2.$$
	Summing up, we get
	\begin{align*}
		\int _X \ch_2 (\cE)	=\int _X \ch_2 (\cG)+\deg _H \cT-\frac {\tau}{2}L_1^2.
	\end{align*}
	After rewriting, using $2r \int _X \ch_2 (\cE)	=  c_1^2(
	\cE)-{\int_X \Delta (\cE)}$, $2r \int _X \ch_2 (\cG)	=  c_1^2(
	\cG)-{\int_X \Delta (\cG)}$ and $c_1 (\cG)= c_1 (\cE)-\tau H$, we get the required formula.
\end{proof}

\section{Boundedness on normal varieties in positive characteristic}

In this section we prove strong restriction theorems for semistable sheaves and we show some boundedness results.  In particular we prove Theorems \ref{Main1} and \ref{Main2}.

Let $X$ be a normal projective variety of dimension $n$ defined over an algebraically closed field $k$ of characteristic $p>0$ and let $L=(L_1,...,L_{n-1})$ be a collection of nef line bundles on $X$. 

\subsection{Slope semistability and its behaviour under pullbacks}\label{Subsection:beta}

For a coherent $\cO_X$-module $\cE$ of positive rank $r$ we define its \emph{slope} with respect to the collection $L$ by
$$\mu_L(\cE):= \frac{c_1(\cE). L_1...L_{n-1}}{r}.$$
Note that by definition $\mu _L(\cE ^{**})= \mu _L(\cE)$.

We say that a coherent $\cO_X$-module $\cE$ is \emph{slope $L$-semistable} if either it is torsion
or for every coherent $\cO_X$-submodule $\cF\subset \cE$
of positive rank we have $\mu_L(\cF)\le \mu_L(\cE)$. We say that $\cE$ is \emph{strongly slope $L$-semistable} if for all $m\ge 0$ the Frobenius pullbacks $(F_X^m)^*\cE$ are slope $L$-semistable.
In this section we usually consider slope semistability with respect to our fixed collection $L$ (unless explicitly stated). So for simplicity of notation we usually ignore dependence of slopes on $L$.

Every coherent $\cO_X$-module $\cE$ of positive rank admits the Harder--Narasimhan filtration 
$0\subset \cE_0\subset \cE_1\subset ...\subset \cE_s=\cE$. It is a unique filtration by coherent 
$\cO_X$-submodules such that $\cE_0$ is torsion, quotients $\cE^i:=\cE_i/\cE_{i-1}$ are torsion free
and slope $L$-semistable  for $i=1,...,s$ and we have
$\mu_1=\mu  (\cE^1)>...>\mu_s=\mu  (\cE^s)$. In the following we write $\mu_{\max} (\cE)$ for $\mu_1$ and  $\mu_{\min} (\cE)$ for $\mu_s$.

The proof of \cite[Theorem 2.7]{La1}  works on normal varieties and it gives the following result:

\begin{Theorem}
	Let $\cE$ be a coherent $\cO_X$-module  of positive rank. Then there exists $m_0$ such that for all $m\ge m_0$ all quotients in the Harder--Narasimhan filtration of $(F_X^{m})^*\cE$ are strongly slope $L$-semistable.
\end{Theorem}

For a coherent reflexive $\cO_X$-module $\cE$ we set
$$L_{\max} (\cE)
=\lim _{m\to \infty} \frac{\mu _{\max} (F_X^{[m]}\cE)}{p^m}$$
and
$$L _{\min} (\cE) =\lim _{k\to \infty} \frac{\mu _{\min} (F_X^{[m]}\cE)}{ p^m}.$$
By the above theorem these numbers are well defined rational numbers.

\medskip

Now choose a nef Cartier divisor $A$ on $X$ such that 
$T_X(A)$ is globally generated. Then \cite[Corollary 2.5]{La1} still holds and it implies the following
lemma.

\begin{Lemma}\label{bounds-for-L_max}
	For any coherent reflexive $\cO_X$-module $\cE$ of rank $r$ we have
	$$\max (L_{\max}(\cE) -\mu (\cE), \mu (\cE) -L_{\min} (\cE))\le AL_1...L_{n-1}. $$
\end{Lemma}

As in \cite{La1} we also set 
$$\beta _r :=
\left( \frac{r(r-1)}{p-1}AL_1\dots L_{n-1}\right) ^2.
$$

\subsection{Restriction theorem and Bogomolov's inequality}\label{Section:Restriction+Bogomolov}

We define an open cone in $N_L (X)$
$$K^+_L=\{ D\in N_L (X): D^2.L_2\dots L_{n-1}>0 \hbox{ and }
D.L_1L_2\dots L_{n-1}\ge 0\} .$$
As in the smooth case, by Lemma \ref{HIT}
this cone  is ``self-dual'' 
in the following sense:
$$D\in K_L^{+} \hbox{ if and only if } D.D'.L_2\dots L_{n-1}>0\hbox{ for all }
D'\in {\overline{K_L^+}}\backslash \{ 0\} .$$

Let us fix a coherent reflexive $\cO_X$-module $\cE$.
Using Lemma \ref{needed-for-boundedness} one can follow the proofs of \cite[Theorems 3.1, 3.2, 3.3 and 3.4]{La1}
and get the following results:

\begin{Theorem}\label{restriction}
Assume that $L_1$ is very ample and
the restriction of $\cE$ to a general divisor $H\in |L_1|$ is not slope
$(L_2|_H,\dots , L_{n-1}|_H)$-semistable. 
Let $r _i$ and $\mu_i$ denote  ranks and  slopes (with respect to $(L_2|_H,\dots , L_{n-1}|_H)$)
 of the Harder--Narasimhan filtration  of $\cE|_H$. Then
$$\sum _{i<j} r_{i} r_{j}(\mu _{i} -\mu _{j})^2\le
d \cdot \int_X\Delta (\cE)L_2\dots L_{n-1}+2r^2(L_{\max}(\cE) -\mu (\cE)) (\mu (\cE) -L_{\min} (\cE)),
$$ 
where $d=L_1^2L_2....L_{n-1}$.
\end{Theorem}

\begin{Theorem}\label{Bog1}
If $\cE$ is strongly slope
$(L_1,\dots, L_{n-1})$-semistable  then 
$$ \int_X \Delta (\cE)L_2\dots L_{n-1}\ge 0.$$ 
\end{Theorem}

\begin{Theorem}\label{Bog2}
If $\cE$ is  slope $(L_1,\dots, L_{n-1})$-semistable then 
$$L_1^2L_2\dots L_{n-1}\cdot
\int_X \Delta (\cE)L_2\dots L_{n-1}+ \beta _r\ge 0.
$$
\end{Theorem}

\begin{Theorem}
If
$L_1^2L_2\dots L_{n-1}\cdot\int_X  \Delta (\cE) L_2\dots L_{n-1}+  \beta _r<0$ 
then there exists a rank $1\le r'<r$ saturated reflexive subsheaf $\cE '\subset \cE$ such that
$\left( \frac{c_1(\cE')}{r'}- \frac{c_1(\cE)}{r}\right)$ lies in $ K^+_L.$ 
\end{Theorem}

The only difference in proofs with respect to \cite{La1} is that one should consider $F_X^{[m]}\cE$ instead of $(F_X^m)^*\cE$. Also in the surface case one needs to  use \cite[Corollary 6.6]{La-Chern} instead of using the arguments of \cite{La1} that do not work for normal surfaces.

\medskip

\subsection{Strong restriction theorems}

In this subsection  we assume that  $d=L_1^2L_2...L_{n-1}>0$. As in \cite[Theorem 5.1]{La1}, Lemma \ref{needed-for-boundedness}  and Theorems \ref{Bog1} and \ref{Bog2} imply the following
Bogomolov's inequality for all reflexive sheaves.

\begin{Theorem}\label{Bogomolov-unstable-inequality}
		If $\cE$ is a coherent reflexive $\cO_X$-module then 
	$$L_1^2L_2\dots L_{n-1}\cdot
\int_X	\Delta (\cE)L_2\dots L_{n-1}+ r^2(L_{\max} (\cE) -\mu  (\cE)) (\mu (\cE) -L_{\min} (\cE) )\ge 0
	$$
	and
	$$L_1^2L_2\dots L_{n-1}\cdot
\int_X	\Delta (\cE)L_2\dots L_{n-1}+r^2(\mu_{\max} (\cE) -\mu (\cE)) (\mu (\cE) -\mu_{\min} (\cE) ) +
	\beta _r \ge 0.$$
\end{Theorem}

This immediately implies the following corollary.

\begin{Corollary}\label{Bogomolov's-inequality}
	Let us fix some positive integer $r$ and some non-negative rational number $\alpha$.
	There exists some constant $C=C(X,L,r,\alpha)$ depending only on $X$, $L$, $r$ and $\alpha$ such that for every  coherent reflexive $\cO_X$-module $\cE$ of rank $r$ with
	$\mu _{\max ,L } (\cE)-\mu _L (\cE)\le \alpha$ we have
	$$\int _X \Delta (\cE) L_2...L_{n-1}\ge C.$$
\end{Corollary}

\medskip

As in \cite{La1} we can use Theorem \ref{Bogomolov-unstable-inequality} to prove strong restriction theorems for reflexive sheaves on normal varieties. We take this opportunity to 
show proof of a stronger result that combines \cite[Theorem 5.2]{La1} with \cite[Theorem 3.7]{La7}.

\begin{Theorem}\label{strong-restriction-stable}
	Let  $\cE$ be a coherent reflexive $\cO_X$-module of rank $r\ge 2$. Assume that $\cE$ is slope $(L_1,...,L_{n-1})$-stable.  Let $H\in |L_1^{\otimes m}|$  be an irreducible normal divisor.  
	\begin{enumerate}
		\item 	 If 
		$$m>\left\lfloor \frac{r-1}{ r}\int_X\Delta (\cE)L_2\dots L_{n-1} +\frac{1}{dr(r-1)}
		+\frac{(r-1)\beta _r}{dr}\right\rfloor$$
		then $\cE|_H$ is  slope $(L_2|_H,...,L_{n-1}|_H)$-stable.
		\item 	If
		$$m\le \left\lfloor \frac{r-1}{ r}\int_X\Delta (\cE)L_2\dots L_{n-1} +\frac{1}{dr(r-1)}
		+\frac{(r-1)\beta _r}{dr}\right\rfloor$$
		then
		\begin{align*}
			\mu _{\max, L_2|_H,...,L_{n-1}|_H }& (\cE|_H)-\mu _{ L_2|_H,...,L_{n-1}|_H } (\cE|_H)\le\\ 
			&\frac{1}{2r}\left( d  \left(\int_X \Delta (\cE )L_2\dots L_{n-1}-\frac{r}{r-1}m
			\right)+\frac{1}{(r-1)^2}+{\beta_r}\right).
		\end{align*}
	\end{enumerate}
\end{Theorem}

\begin{proof}
	Let $i: H\hookrightarrow X$ denote the closed embedding.
	Since $\cE$ is reflexive, $i^*\cE$ is a torsion free $\cO_H$-module.
	Let  $\cS$ be a saturated subsheaf of $i^*\cE$ 
	of rank $\rho$. Set $\cT:=(i^*\cE)/\cS$ and let $\cE '$ be the kernel of the composition
	$\cE\to i_*(i^*\cE)\to i_*\cT$. Since $\cT$ is a torsion free $\cO_H$-module, $\cE '$ is a coherent reflexive $\cO_X$-module and we have two short exact sequences:
	$$0\to \cE '\to \cE \to i_*\cT\to 0$$
	and
	$$0\to \cE(-H)\to \cE ' \to i_* \cS\to 0.$$
	Lemma \ref{Delta-computation} implies that
	\begin{align*}
		\int_X \Delta (\cE ')L_2\dots L_{n-1}
		=&\int_X \Delta (\cE)L_2\dots L_{n-1}-
		\rho  (r-\rho ) H^2L_2\dots L_{n-1}\\
		&+2 (r c_1(\cT)-(r-\rho ) c_1(i^*\cE)). L_2|_H\dots L_{n-1}|_H.
	\end{align*}
	Since $\cE'\subset \cE$ and $\cE$ is slope $(L_1,...,L_{n-1})$-stable
	we have
	$$
	\mu _{\max} (\cE')-\mu (\cE')=\frac{r-\rho}{r}HL_1...L_{n-1}+\mu _{\max} (\cE')-\mu ( \cE)
	\le \frac{r-\rho}{r}md
	-\frac{1}{r(r-1)} .
	$$ Similarly, since $\cE(-H)\subset \cE'$ we have
	$$
	\mu (\cE')-\mu _{\min} (\cE')=\frac{\rho}{r}HL_1...L_{n-1}+\mu (\cE(-H))- \mu _{\min} (\cE')
	\le \frac{\rho}{ r}md-\frac{1}{r(r-1)} .$$ 
	So Theorem \ref{Bogomolov-unstable-inequality} gives
	\begin{align*}
		0\le& 	d\cdot  \int_X \Delta (\cE ')L_2\dots L_{n-1} + r^2(\mu_{\max }(\cE') -\mu  (\cE'))(\mu  (\cE') -\mu_{\min } (\cE') )+\beta_r \\
		\le & 	d\cdot \int_X \Delta (\cE )L_2\dots L_{n-1} -\rho (r-\rho) m^2d^2- 2r\rho (\mu (\cS)-\mu(i^*\cE))\\ 
		&+
		\left(\rho md -\frac{1}{r-1}\right)
		\left ((r-\rho)md -\frac{1}{r-1}\right) +\beta_r.\\
	\end{align*}
	If $\mu (\cS)\ge \mu( i^*\cE)$ then we get
	\begin{align*}
		\frac{2(r-1)}{d} (\mu (\cS)-\mu(i^*\cE))+m\le 	\frac{ r-1}{r} \int_X \Delta (\cE )L_2\dots L_{n-1}+\frac{1}{dr(r-1)}+\frac{(r-1)\beta_r}{dr},
	\end{align*}
	which implies the required assertions.	
\end{proof}

\begin{Remark}
	The above proof works also for an arbitrary irreducible normal divisor $D\subset X$, which is nef and Cartier.  In this way one gets restriction theorems taking into account
the difference of directions of lines given by classes of $D$ and $L_1$ in $N_L(X)$. We leave the details of proof to the interested reader.
\end{Remark}

\medskip

As in \cite[Corollary  5.4]{La1}  the above theorem together with Lemma \ref{needed-for-boundedness} implies the following result:

\begin{Corollary}\label{Bogomolov-restriction}
	Let  $\cE$ be a coherent reflexive $\cO_X$-module of rank $r\ge 2$. Assume that $\cE$ is slope $(L_1,...,L_{n-1})$-semistable 	and let $H\in |L_1^{\otimes m}|$  be an irreducible normal divisor. 
Let $0= \cE_0\subset \cE_1\subset ...\subset \cE_s=\cE$ be a
	Jordan--H\"older filtration of $\cE$ and let us assume that all $(\cE_i/\cE_{i-1})|_H$ are torsion free.
	\begin{enumerate}
	\item 	If 
	$$m>\left\lfloor \frac{r-1}{ r}\int_X\Delta (\cE)L_2\dots L_{n-1} +\frac{1}{dr(r-1)}
	+\frac{(r-1)\beta _r}{dr}\right\rfloor$$
	then $\cE|_H$ is slope $(L_2|_H,...,L_{n-1}|_H)$-semistable.
	\item 	If
		$$m\le \left\lfloor \frac{r-1}{ r}\int_X\Delta (\cE)L_2\dots L_{n-1} +\frac{1}{dr(r-1)}
		+\frac{(r-1)\beta _r}{dr}\right\rfloor$$
		then
		\begin{align*}
			\mu _{\max} (\cE|_H)\le \mu (\cE)+
			\frac{1}{2r}\left( d  \left(\int_X \Delta (\cE )L_2\dots L_{n-1}-\frac{r}{r-1}m
			\right)+\frac{1}{(r-1)^2}+{\beta_r}\right).
		\end{align*}
	\end{enumerate}
\end{Corollary}

\begin{proof} Note that all $\cE_i$ are reflexive as they are saturated in 
	$\cE$.  Let us set $\cF_i:=(\cE_i/\cE_{i-1})^{**}$ and $r_i=\rk \cF_i$.
By	Lemma \ref{needed-for-boundedness} we have
	\begin{align*}
		\frac{\int _X \Delta  (\cE)L_2...L_{n-1}}{r}\ge \sum _i \frac{\int _X \Delta  (\cF_i)L_2...L_{n-1}}{r_i}.
	\end{align*}
In the first case either $r_i=1$ or $r_i\ge 2$ and
$$m>\left\lfloor \frac{r_i-1}{r_i}\int_X\Delta (\cF_i)L_2\dots L_{n-1} +\frac{1}{dr_i(r_i-1)}
+\frac{(r_i-1)\beta _{r_i}}{dr_i}\right\rfloor .$$
Note that in the above inequality we need to worry about the term 
$\frac{1}{dr_i(r_i-1)}$, which can be larger than $\frac{1}{dr(r-1)}$. However, the difference is compensated by the other terms unless both of them are 0 in which case $\lfloor \frac{1}{dr_i(r_i-1)}\rfloor =0$.
Applying Theorem \ref{strong-restriction-stable} we see that $\cF_i|_H$ is stable
with the same slope as $\cE|_H$.
Since $(\cE_i/\cE_{i-1})|_H$ are torsion free, the sequences
$$0\to \cE_{i-1}|_H\to \cE_i|_H\to \cF_i|_H$$
are exact. Now a simple induction show that all $\cE_i|_H$ are  slope $(L_2|_H,...,L_{n-1}|_H)$-semistable.

The second case is completely analogous. We just need to use the fact that
$$	\mu _{\max} (\cE|_H)\le	\max _i\mu _{\max} (\cF _i|_H).$$
\end{proof}

\begin{Remark}
	The above results give also restriction theorems for torsion free sheaves. More precisely,
	let $\cE$ be a coherent torsion free $\cO_X$-module, which is slope $(L_1,...,L_{n-1})$-(semi)stable. Then its reflexive hull $\cE^{**}$ is also slope $(L_1,...,L_{n-1})$-(semi)stable, so we can apply Theorem \ref{strong-restriction-stable}
	and Corollary \ref{Bogomolov-restriction} to  $\cE^{**}$. If $\cE|_H$ is torsion free then it is slope $(L_2|_H,...,L_{n-1}|_H)$-(semi)stable if (and only if) $\cE^{**}|_H$ is slope $(L_2|_H,...,L_{n-1}|_H)$-(semi)stable.	
\end{Remark}

\subsection{Boundedness}

In this subsection we assume that $n\ge 2$ and  all line bundles $L_1,...,L_{n-1}$ are ample.  
For $\alpha\in B^1(X)$ we write $\alpha\sim 0$ if $\alpha.L_1...L_{n-1}=0$ and
$$\alpha^2.L_1...\widehat{L_i}...L_{n-1}=0$$
for $i=1,...,n-1$. Note that the subset $S:=\{\alpha\in B^1(X) : \alpha\sim 0\}$ forms a $\ZZ$-submodule of $B^1(X)$.
Indeed, Lemma \ref{easy-HIT} implies that $\alpha \in S$ if and only if  the class of $\alpha$ in $N_{(L_1...\widehat{L_i}...L_{n-1})}(X)$ vanishes for
$i=1,...,n-1$. So $S$ is the intersection of kernels of quotient maps
$$B^1(X)\to N_{(L_1...\widehat{L_i}...L_{n-1})}(X).$$
In the following we  set $C^1(X; L_1,...,L_{n-1}):=B^1(X)/S$. 

\begin{Lemma}\label{injection}
$C^1(X;L_1,...,L_{n-1})$ is a free $\ZZ$-module of finite rank. Moreover, for
a general divisor $H\in |L_1|$ we have a well-defined Gysin homomorphism 
$$C^1(X;L_1,...,L_{n-1})\to C^1(H;L_2|_H,...,L_{n-1}|_H).$$
\end{Lemma}

\begin{proof} The first assertion follows from the definition and Lemma \ref{Neron-Severi}.
By \cite[Proposition 10.3]{Fu} we have a Gysin homomorphism $B^1(X)\to B^1(H)$. Let us denote the image of  
$\alpha \in B^1(X)$ in $B^1(H)$ by $\alpha |_H$. It is sufficient to show that if $\alpha\sim 0$ in $B^1(X)$ then $\alpha|_H\sim 0$ in $B^1(H)$. Since $S\subset B^1(X)$ is a finitely generated $\ZZ$-submodule, it is sufficient to check this condition for finitely finitely many generators of $S$. But for a general divisor  $H\in |L_1|$ 
we have 
$$\alpha|_H^2.L_2|_H...\widehat{L_i|_H}...L_{n-1}|_H=\alpha^2.L_1...\widehat{L_i}...L_{n-1}=0$$
for $i=1,...,n-1$. We also have $\alpha|_H.L_1|_H...L_{n-1}|_H=\alpha.L_1...L_{n-1}=0$, which proves that
the class of  $\alpha|_H$ in $C^1(H;L_2|_H,...,L_{n-1}|_H)$ vanishes.
\end{proof}

\medskip

Corollary \ref{Bogomolov-restriction}  can be used to prove the following boundedness result, which
does not immediately follow from Kleiman's criterion. Part of the proof follows the idea of proof of \cite[Theorem 3.4]{La-PSPUM}.

\begin{Theorem} \label{boundedness}
Let us fix some classes $c_{1}\in C^1(X;L_1,...,L_{n-1})$, a positive integer $r$  and some real numbers $c_2$ and $\mu _{\max}$.  
Let $\cA$ be the set of reflexive coherent $\cO_X$-modules $\cE$ such that 
\begin{enumerate}
\item $\cE$ has  rank $r$, 
\item the class of $c_1(\cE)$ in $C^1(X;L_1,...,L_{n-1})$ is equal to $c_{1}$,
\item $\int _X c_2 (\cE) L_1...\widehat{L_i}...L_{n-1} \le c_2$ for $i=1,...,n-1$,
\item $\mu _{\max } (\cE)\le \mu _{\max}$.
\end{enumerate}
Then the set $\cA$ is bounded.
\end{Theorem}

\begin{proof} For $n=2$ the assertion is well-known (see, e.g.,  \cite[Theorem 4.4]{La1}), so we can assume that $n\ge 3$. Without loss of generality we can also  assume that  all $L_i$ are very ample.
Note that if $k\subset K$ is an algebraically closed field extension then the set $\cA$ is bounded if and only if the set $\cA_K:= \{ \cE_K: \cE \in \cA\}$ of sheaves on $X_K$ is bounded. This follows from 
\cite[Lemma 1.7.6]{HL} and the fact that the Castelnuovo--Mumford regularity of $\cE$ (with respect to some fixed very ample line bundle $\cO_X(1)$) coincides with the Castelnuovo--Mumford regularity of $\cE$ (here we use the fact that $H^i(X_K, \cE(j)_K)= H^i(X, \cE (j))\otimes _kK$).
So by  Theorem \ref{properties-of-ch_2}, (2) and an analogue of \cite[Theorem 1.3.7]{HL} for slope semistability, we can assume that the base field $k$ is uncountable.

For fixed $\cE\in \cA$  and  for a very general divisor $H\in |L_1|$,
the following conditions are satisfied:
\begin{enumerate}
	\item $\cE|_{H}$ is reflexive as an $\cO_H$-module (by  \cite[Corollary 1.1.14]{HL}),
	\item for $i=2,...,n-1$ we have $$\int _H c_1 (\cE |_{H})^2 L_2|_H...\widehat{L_i|_H}...L_{n-1}|_H=\int _X c_1 (\cE )^2 L_1...\widehat{L_i}...L_{n-1}=c_1^2 .L_1...\widehat{L_i}...L_{n-1}$$
	(by Theorem \ref{main1}),
	\item for $i=2,...,n-1$ we have
	$$\int _H c_2 (\cE |_{H}) L_2|_H...\widehat{L_i|_H}...L_{n-1}|_H=\int _X c_2 (\cE ) L_1...\widehat{L_i}...L_{n-1}\le c_2$$
	(by Theorem \ref{properties-of-ch_2}),
	\item  if $0= \cE_0\subset \cE_1\subset ...\subset \cE_s=\cE$ be the Harder--Narasimhan filtration of 
	of $\cE$ then all  $(\cE_i/\cE _{i-1})|_H$ are torsion free and
	restriction of any quotient of a Jordan--H\"older filtration of $\cF_i:= (\cE_i/\cE _{i-1})^{**}$ to $H$ is torsion free  as an $\cO_H$-module (by  \cite[Corollary 1.1.14]{HL}).
\end{enumerate}
Let us fix a normal hypersurface $H\in |L_1|$ such that the Gysin homomorphism 
$C^1(X;L_1,...,L_{n-1})\to C^1(H;L_2|_H,...,L_{n-1}|_H)$ is well defined (see Lemma \ref{injection}). 
Let us consider the set $\cA _H$ of all sheaves $\cE\in \cA $  that satisfy the above conditions (1)--(4).
By Corollary \ref{Bogomolov's-inequality} there exists $C$ such that for every slope $(L_1,...,L_{n-1})$-semistable reflexive $\cO_X$-module $\cF$ of rank $\le r$ we have
	$$\int _X \Delta (\cF) L_2...L_{n-1}\ge C.$$
Let us set $r_i=\rk \cF_i$ and $\mu_i:= \mu _L(\cF_i)$. 
By	Lemma \ref{needed-for-boundedness} and \cite[Lemma 1.4]{La1} we have
\begin{align*}
& \frac{\int _X \Delta  (\cF_i)L_2...L_{n-1}}{r_i} \le \sum _j \frac{\int _X \Delta  (\cF_j)L_2...L_{n-1}}{r_j}-(s-1)C\\
 &\le 	\frac{\int _X \Delta  (\cE)L_2...L_{n-1}}{r}+	\frac{1}{rd}\sum _{i<j} r_ir_j\left(\mu_i -  \mu_j\right)^2-(s-1)C\\
&\le 	\frac{ 2r c_2-(r-1)c_{1}^2.L_2...L_{n-1}}{r}+	\frac{r}{d} \left(\mu_{\max}(\cE) -  \mu (\cE)\right) \left(\mu(\cE) -  \mu _{\min} (\cE)\right)-(s-1)C
\end{align*}
Since 
$$\mu(\cE) -  \mu _{\min} (\cE) \le (s-1) \left(\mu_{\max}(\cE) -  \mu (\cE)\right)\le (s-1) \left(\mu_{\max}-\frac{1}{r}c_{1}.L_1...L_{n-1}\right) ,$$
Corollary \ref{Bogomolov-restriction} and condition (4) imply that for all $i$ we have
$\mu _{\max} (\cF_i|_{H})  \le C_1$
for some $C_1$ that depends only on $r$, $c_1$, $c_2$ and $\mu_{\max}$.
Condition (4) implies also that the sequences
$$0\to \cE_{i-1}|_H\to \cE_i|_H\to \cF_i|_H$$
are exact, so
$$	\mu _{\max} (\cE|_{H})\le \max _i	\mu _{\max} (\cF_i|_{H}) \le C_1.$$
For any $\cE\in \cA$ the class of $c_1(\cE|_H)$ in $C^1(H;L_2|_H,...,L_{n-1}|_H)$ coincides with $c_1|_H$.
This follows from the fact that $\cE|_H$ is torsion free, so also locally free on a big open subset of $H$.
So already the class of $c_1(\cE|_H)$ coincides with $c_1(\cE)|_H\in A^1(H)$.
By the induction assumption this implies that the set of sheaves $\{\cE|_H\}_{\cE\in \cA_H}$ is bounded.
To simplify notation we write $\cO_X(1)=L_1$.
There exists some integers $a$, $b$ and $c$ such that for all $\cE\in \cA_H$  the following conditions are satisfied:
\begin{enumerate}
	\item $H^i (H, \cE |_H(m))=0$ for all $m\ge a$ and all $i>0$,
	\item $H^1 (H, \cE |_H(-m))=0$ for $m\ge b$,
	\item  $h^1 (H, \cE |_H(m))\le c$ for all $m$.
\end{enumerate}
The second condition above follows from the well-known Enriques--Severi--Zariski lemma for reflexive sheaves on normal varieties (see \cite[Lemma 0FD8]{SP}). 
Let us fix $\cE\in \cA_H$.
For all $m\in \ZZ$ we have short exact sequences
$$0\to \cE({m-1})\to \cE (m)\to \cE(m)|_{H} \to 0.$$
Let us take $i\ge 2$. Then for all $m\ge a$ we have $H^i (X, \cE (m-1))=H^i (X, \cE (m))$. So by Serre's vanishing theorem we get $H^i (X, \cE (m-1))=0$ for $m\ge a$.
We also know that $h^1 (X, \cE (m))\le h^1 (X, \cE (m-1))$ for all $m\ge a$. 
Let us consider the embedding $X\hookrightarrow \PP^N$ given by the linear system $|\cO_{X}(1)|$
and let $\tilde H\subset \PP^N$ be the hyperplane defining $H$. For any $m\in \ZZ$ we have a commutative diagram
$$\xymatrix{
	H^0(\PP^N, \cE (m))\otimes 	H^0(\PP^N, \cO_{\PP^N} (1))\ar[d]^{\beta_1}\ar[r]^-{\alpha_1}& 	H^0(\PP^N, \cE (m+1))\ar[d]^{\beta_2}\\
	H^0(\tilde H, \cE|_H (m))\otimes 	H^0(\tilde H, \cO_{\tilde H} (1))\ar[r]^-{\alpha_2}& 	H^0(\tilde H, \cE|_H (m+1)).\\
}
$$
Assume that $h^1 (X, \cE (m))= h^1 (X, \cE (m-1))$ for some $m\ge a+n-1$. Then the map $\beta_1$ in the above diagram is surjective. Since $m\ge a+n-1$, we also know that $H^i (H, \cE|_H(m-i))=0$ for $i>0$ so by the Castelnuovo--Mumford theorem the map $\alpha_2$ is also surjective. It follows that $\beta _2$ is surjective. This implies that $h^1 (X, \cE (m+1))= h^1 (X, \cE (m))$. So by Serre's vanishing theorem we see that  $h^1 (X, \cE (m))=0$.
This shows that for  $m\ge a+n-2$ the sequence $\{h^1 (X, \cE (m)) \}$ is strictly decreasing until it reaches $0$. So $h^1 (X, \cE (l))=0$ for $l\ge h^1 (X, \cE (a+n-2))+a+n-2$. 
Since $\cE$ is reflexive and $X$ is normal we know that $h^1 (X, \cE (-l))=0$ for $l\gg 0$ (here we again use \cite[Lemma 0FD8]{SP}). So for all $m\in \ZZ$ we have
\begin{align*}
h^1 (X, \cE (m))&\le h^1 (X, \cE (m-1))+ h^1 (H, \cE |_H(m))\le h^1 (X, \cE (m-2))+ h^1 (H, \cE |_H(m))\\
&+ h^1 (H, \cE |_H(m-1))\le ...\le \sum _{l\ge 0} h^1 (H, \cE |_H(m-l))\le 	(m+b)c.
\end{align*}
In particular,  $h^1 (X, \cE (a+n-2))\le (a+n-2+b)c$. Therefore $h^1 (X, \cE (m))=0$ for 
$m\ge (a+n-2)(c+1)+bc$. This shows that there exists a constant $m_0$ such that all $\cE\in \cA_H$ are $m$-regular for all $m\ge m_0$ and hence $\cA_H$ is a bounded family (see \cite[Lemma 1.7.6]{HL}). Since the family of divisors $H\in |L_1|$ is bounded, this also gives boundedness of the family $\cA$.
\end{proof}

\begin{Remark}
If $L_1=...=L_{n-1}$ then the above theorem follows from \cite[Theorem 4.4]{La1}  (see \cite[Theorem 6.4]{La-Chern}).  In general, the problem is that we need restriction theorems for multipolarizations on normal varieties and the method of proof of   \cite[Theorem 4.4]{La1} for singular varieties depends on the projection method that works only if we have one polarization. In case of characteristic zero, restriction theorems needed for multipolarizations follow easily from the results of \cite{La1} by passing to the resolution of singularities and using Bertini's theorem. Unfortunately, this method also does not work for varieties defined over a field of positive characteristic. However, even in this case Theorem \ref{boundedness} is new.
\end{Remark}

As in \cite[Corollary 6.7]{La-Chern}, Corollary \ref{Bogomolov's-inequality} and the above theorem imply the following result. 

\begin{Corollary} \label{boundedness2}
Let us fix some positive integer $r$, integer $\ch_1$ and some real
numbers $\ch _2$ and $\mu _{\max}$. 
Let $\cB$ be the set of reflexive coherent $\cO_X$-modules $\cE$ such that 
\begin{enumerate}
\item $\cE$ has  rank $r$, 
\item  $\int _X \ch_1 (\cE) L_1...L_{n-1}=\ch_1$,
\item $\int _X \ch_2 (\cE) L_1...\widehat{L_i}...L_{n-1} \ge \ch_2$ for $i=1,...,n-1$,
\item $\mu _{\max ,L } (\cE)\le \mu _{\max}$.
\end{enumerate}
Then the set $\cB$ is bounded.
\end{Corollary}

\begin{proof}
By Corollary \ref{Bogomolov's-inequality} there exists a constant $C$ 
such that for all  $\cE \in \cB$ we have
$$C\le \int _X \Delta (\cE) L_1...\widehat{L_i}...L_{n-1}=  c_1 (\cE)^2.L_1...\widehat{L_i}...L_{n-1} - 2r \int _X \ch_2 (\cE). L_1...\widehat{L_i}...L_{n-1} $$
for $i=1,...,n-1$.
Therefore $c_1 (\cE)^2.L_1...\widehat{L_i}...L_{n-1}\ge C+2r \,\ch _2$. Let us write $[c_1(\cE)]=\alpha _i [L_i]+D_i\in N^1_{(L_1,...,\widehat{L_i},...,L_{n-1})}(X)$, where $\alpha _i= \frac{\ch _1}{L_1...L_i^2...L_{n-1}}$.
Then $D_i.L_1...L_{n-1}=0$ and 
$$c_1 (\cE)^2.L_1...\widehat{L_i}...L_{n-1}=\alpha _i^2 +D_i^2.L_1...\widehat{L_i}...L_{n-1}.$$ 
Therefore
$$D^2_i.L_1...\widehat{L_i}...L_{n-1}\ge C+2r \,\ch _2 - \alpha _i ^2\cdot L_1...L_i^2...L_{n-1}.$$
But by the Hodge index theorem (see Lemma \ref{easy-HIT}) the intersection form is negative definite on $L_i^{\perp}\subset N^1_{(L_1,...,\widehat{L_i},...,L_{n-1})}(X)$, so there are only finitely many possibilities for $D_i\in N^1_{(L_1,...,\widehat{L_i},...,L_{n-1})}(X)$.
Since the canonical map 
$$C^1(X;L_1,...,L_{n-1})\to \bigoplus _{i=1}^{n-1} N^1_{(L_1,...,\widehat{L_i},...,L_{n-1})}(X)$$
is injective, there are also only finitely many possibilities for the classes $[c_1(\cE)]\in C^1(X;L_1,...,L_{n-1})$. Now the assertion follows from Theorem \ref{boundedness}.
\end{proof}

The above corollary has some nontrivial implications even in the rank one case:

\begin{Corollary}
The canonical map $\bar B^1(X)\to C^1(X;L_1,...,L_{n-1})$ is an isomorphism.
\end{Corollary}

\begin{proof}
Let $D$ be a Weil divisor representing the class in the kernel of $B^1(X)\to C^1(X;L_1,...,L_{n-1})$.  Corollary \ref{boundedness2} implies that the set $\{ \cO_X(mD)\}_{m\in \ZZ}$ is bounded. So
the set $\{ [mD] \} _{m\in \ZZ}$ of the corresponding classes in $B^1(X)$ is finite. Therefore  $[D]=0\in \bar B^1(X)$. 
\end{proof}

\begin{Remark}
The above corollary implies that some multiple of a Weil divisor $D$ on $X$ is algebraically equivalent to $0$
if and only if $D.L^{n-1}=D^2.L^{n-2}=0$ for some ample line bundle $L$. This allows to
give generalization of \cite[Theorem 9.6.3]{Kl} to rank $1$ reflexive sheaves  on normal projective varieties.	
\end{Remark}

\section{Modules over Lie algebroids and Higgs sheaves}\label{Section:Higgs-sheaves}

In this section we show various definitions and simple results on modules over Lie algebroids and on generalized Higgs sheaves. We also compare our notion with the one used in \cite{GKPT2}. We finish the section with definition of semistability and with a restriction theorem for generalized Higgs sheaves. In the whole section $X$ is a scheme over some fixed field $k$.

\subsection{Basic definitions}

Let us recall that a tangent sheaf $T_{X/k}$ is defined as $\cHom_{\cO_X} (\Omega _{X/k}, \cO_X)$. 
For every open subset $U\subset X$ we have a canonical isomorphism 
$$T_{X/k} (U)= \Hom _{\cO_U} (\Omega _{X/k}|_U, \cO_U)\to \mathop{\rm Der}\, _k(\cO_U, \cO_U), \quad \delta \to \delta \circ d_{X/k}.$$
So in the following we will identify sections of $T_X/k$ with $k$-derivations of the structure sheaf
without mentioning it. Now let us  recall the following definitions (see \cite[Sections 2 and 3]{La2}).

\begin{Definition}
	A \emph{Lie algebroid on $X/k$} is a triple $(L, [\cdot,
	\cdot ]_L, \alpha )$ consisting of 
\begin{enumerate}
	\item a quasi-coherent $\cO_X$-module $L$,
	\item a morphism of sheaves of $k$-vector spaces $ [\cdot,
	\cdot ]_L: L\otimes _kL\to L$,
	\item an $\cO_X$-linear map $\alpha: L\to
	T_{X/k}$, $x\to \alpha_x$, called \emph{anchor map},
\end{enumerate}	
	such that the following conditions are satisfied:
	\begin{enumerate}
		\item 	$(L, [\cdot, \cdot ]_L)$ is a sheaf of $k$-Lie algebras,
		\item $\alpha$ is a homomorphism  of sheaves of $k$-Lie algebras,
		\item We have 
		$$[x,fy]_L=\alpha _x (f)y+f\,[x,y]_L$$
	for all local sections $f\in \cO _X$ and $x,y\in L$.	
	\end{enumerate}
\end{Definition}

\begin{Definition}
Let $L$ be a Lie algebroid on $X/k$. An $L$-module is a pair $(\cE , \nabla)$ consisting of 
a quasi-coherent $\cO_X$-module $\cE$ and an $\cO_X$-linear map of left $\cO_X$-modules
	 $\nabla: L\to \End _{k}\cE$, which is also a map of sheaves of
	$k$-Lie algebras and which satisfies Leibniz's rule
	$$\nabla (x) (fe)=\alpha _x(f)e+\nabla(fx)(e)$$
	for all local sections $f\in \cO_X$, $x\in L$ and $e\in \cE$.
\end{Definition} 

In the above definition (and also below) we use the following identifications. For every open subset $U\subset X$  we have a map $\nabla (U): L(U)\to (\End _k \cE) (U)= \Hom _k (\cE |_U, \cE|_U)$. So every $x\in L (U)$ gives a $k$-linear map $\cE |_U\to \cE|_U$ that is denoted by $\nabla (x)$.
Now the above Leibniz's rule should be interpreted as equality for all open  $V\subset U\subset X$
for $x\in L(U)$, $f\in \cO_X(V)$ and $e\in \cE(V)$, where $fx$ denotes $f(x|_V)$.

\begin{Definition}
	Let $(\cE _1, \nabla _1)$ and $(\cE _2, \nabla _2)$ be $L$-modules for some Lie algebroid $L$.
	A \emph{morphism of $L$-modules} $\varphi : (\cE _1, \nabla _1)\to  (\cE _2, \nabla_2) $ is 
	an $\cO_X$-linear map $\varphi: \cE_1\to \cE_2$ such that for every open subset $U\subset X$ and
	 every $x\in L (U)$	  the diagram
	$$\xymatrix{\cE_1|_U\ar[r]^-{\nabla_1(x)}\ar[d]^{\varphi |_U}& \cE_1|_U\ar[d]^{\varphi |_U}\\
	 \cE_2|_U\ar[r]^-{\nabla_2 (x)}&\cE_2|_U\\
	}$$
	is commutative. 
\end{Definition}

The above definitions allow us to talk about the category $ \Mod{L} (X)$ of $L$-modules.

\subsection{Extensions of modules over Lie algebroids}\label{extensions-of-L-modules}

Let $X$ be an integral normal locally Noetherian scheme over some field $k$ and let $L$ be a Lie algebroid on $X/k$, whose underlying $\cO_X$-module is  coherent and reflexive. By abuse we will call such Lie algebroid \emph{reflexive}.

Let $j: U\hookrightarrow X $ be a big open subscheme $X$.  Let $(\cE, \nabla : L_U\to \End _{k} \cE)$ be an $L_U$-module. By assumption $L=j_*L_U$, so we can set $j_*(\cE, \nabla):=(j_*\cE, j_*\nabla)$, where $j_*\nabla$
acts as $\nabla$ on the sections of $j_*\cE$ (which are the same as sections of $\cE$). 
In this way we can define the functor $j_*:  \Mod{L_U} (U)\to  \Mod{L} (X)$. 

We say that an $L$-module  $(\cE , \nabla)$ is \emph{reflexive} if $\cE$ is coherent and reflexive
as an $\cO_X$-module. By 
$\refMod{L} (X)$ we denote the full subcategory of $ \Mod{L} (X)$, whose objects are 
reflexive $L$-modules. Note that after restricting to reflexive modules  
$j_*$ and $j^*$ define adjoint equivalences of categories $\refMod{L _U} (U)$ and $\refMod{L} (X)$.

\subsection{Tensor operations on modules over Lie algebroids}\label{tensors}

If $\cE_1$ and $\cE_2$ are $L$-modules then we can define natural $L$-module structures on $\cE_1\otimes _{\cO_X} \cE_2$ and on $\cHom_{\cO_X} (\cE_1, \cE_2)$. In particular, if $\cE$ is an $L$-module then
$\cE^{*}$ has a canonical $L$-module structure.

If
$\nabla _1: L\to \End _{k} \cE _1$ and $\nabla _2: L\to \End _{k} \cE _2$ 
are $L$-module structures on $\cE_1$ and $\cE_2$ then we define an $L$-module structure 
$\nabla: L\to \End _{k} (\cE_1\otimes _{\cO_X} \cE_2)$ on $ \cE_1\otimes _{\cO_X} \cE_2$ by the formula
$$\nabla (x)=\nabla_1 (x) \otimes \id + \id \otimes \nabla _2(x).$$

Similarly, we define  the $L$-module structure $\nabla: L\to \End _{k} (\cHom_{\cO_X} (\cE _1, \cE_2))$ on $\cHom_{\cO_X} (\cE _1, \cE_2)$
by the formula $$(\nabla (x)) (\psi)=(\nabla_2 (x))\circ \psi -\psi \circ  (\nabla_1(x)).$$

This shows that for any $L$-module $\cE$ we can define a natural $L$-module structure on $\cE ^*=\cHom_{\cO_X} (\cE , \cO_X)$. So we can also define a reflexive hull $\cE ^{**}$ of an $L$-module.
As in the case of $\cO_X$-modules (see Subsection \ref{reflexive-sheaves})
the inclusion functor $\refMod{L} (X)\to \Mod{L} (X) $ comes with a left adjoint $(\cdot )^{**}:  \Mod{L} (X) \to \refMod{L} (X)$ given by the reflexive hull. 
In particular, we have a natural map $\cE\to \cE ^{**}$ of $L$-modules coming from the adjoint map to the identity on 
$ \cE ^{**}$.

\medskip

\begin{Remark}
	Note that if we have two Higgs sheaves
	$(\cE_1, \theta _1: \cE _1\to \cE _1{\otimes} \Omega_X^{[1]})$ and $(\cE_2, \theta _2: \cE _2\to \cE _2{\otimes} \Omega_X^{[1]})$
	in the sense of \cite{GKPT2} then we can define the Higgs sheaf structure on 
	$\cE_1\otimes _{\cO_X} \cE_2$ but not on $\cHom_{\cO_X} (\cE_1, \cE_2)$. In particular, if $\Omega_X^{[1]}$ is not locally free one cannot 
	define the dual of a  Higgs sheaf as a Higgs sheaf. So we cannot also take a reflexivization of a (torsion free) Higgs sheaf in the sense of \cite{GKPT2}. This is one of the main reasons why we need to use a different definition of a Higgs sheaf.
\end{Remark}

\subsection{Generalized Higgs sheaves}

Let $X$ be a scheme and let $L$ be a quasi-coherent $\cO_X$-module. We can equip $L$ with a trival Lie algebroid structure with zero Lie bracket and zero anchor map.

\begin{Definition}
	An \emph{$L$-Higgs sheaf} is an $L$-module for the trivial Lie algebroid structure on $L$. In other words, it is a pair $(\cE, \theta)$ consisting of a quasi-coherent sheaf $\cE$ of $\cO_X$-modules and an $\cO_X$-linear  map $\theta : L\to \End _{\cO_X} (\cE) $ of sheaves of Lie rings, where $L$ is equipped with the trivial  Lie bracket. 
An $L$-Higgs sheaf $(\cE, \theta)$ is \emph{reflexive} if 
$\cE$ is a coherent reflexive $\cO_X$-module.
\end{Definition}

\medskip

On any scheme $X$, if $\cE$, $\cF$, $\cG$ are sheaves of $\cO_X$-modules then we have a functorial isomorphism of $\Gamma (X, \cO_X)$-modules
$$\Hom _{\cO_X} (\cE\otimes _{\cO_X} \cF, \cG)\mathop{\longrightarrow}^{\simeq}\Hom _{\cO_X}
(\cE, \cHom_{\cO_X}(\cF, \cG)).$$
In particular, we have an isomorphism
$$\alpha: \Hom _{\cO_X} (L\otimes _{\cO_X} \cE, \cE)\mathop{\longrightarrow}^{\simeq}\Hom _{\cO_X}
(L, \End _{\cO_X} (\cE)).$$
This shows that we can replace $\theta$ by an $\cO_X$-linear map $L\otimes _{\cO_X}\cE\to \cE$ that by abuse of notation will also be denoted by $\theta$.

\medskip

Note that we have a map $ L\otimes _{\cO_X}L \to L\otimes _{\cO_X}L$ given by sending $x\otimes y$ to $x\otimes y-y\otimes x$. Since it maps $x\otimes x$ to $0$, this map factors through the canonical projection  $ L\otimes _{\cO_X}L\to {\bigwedge}^2L$. Hence we get the map $\iota: {\bigwedge}^2L\to  L\otimes _{\cO_X}L$ fitting into an exact sequence
$${\bigwedge}^2L\to  L\otimes _{\cO_X}L \to \Sym ^2 L\to 0,$$ 
where the second map is the canonical projection.

The following lemma explains how to check when an $\cO_X$-linear map $L\otimes _{\cO_X}\cE\to \cE$ gives rise to an $L$-Higgs sheaf.

\begin{Lemma}\label{equivalent-Higgs}
	Let us fix an $\cO_X$-linear  map $\theta : L\otimes _{\cO_X} \cE\to \cE$.
	Then the following conditions are equivalent:
	\begin{enumerate}
		\item The composition
		$$\xymatrixcolsep{3pc}\xymatrix{
			{\bigwedge}^2L\otimes _{\cO_X}	\cE\ar[r]^-{\iota \otimes \id }& L\otimes _{\cO_X}L\otimes _{\cO_X} \cE \ar[r]^-{\id{\otimes }\theta}&
			L\otimes _{\cO_X}\cE\ar[r]^-{\theta }&
			\cE\\}$$
		vanishes.	
		\item
		The map $\bar\theta:=\alpha (\theta):L\to \End _{\cO_X} (\cE) $ is a homomorphism of sheaves of Lie rings.
		\item The map $\theta$ extends to a $\Sym ^{\bullet}L$-module structure on $\cE$, i.e., there exists an $\cO_X$-linear  map $\tilde\theta : \Sym ^{\bullet} L\otimes _{\cO_X} \cE\to \cE$ such that 
		$\tilde \theta |_{\cO_X}: \cO_X\otimes _{\cO_X} \cE\to \cE$ is the identity, 	$\tilde \theta |_L=\theta$  
		and the diagram
		$$
		\xymatrix{
			\Sym ^{\bullet} L\otimes _{\cO_X} \Sym ^{\bullet} L\otimes _{\cO_X} \cE\ar[r]^-{\id\otimes \tilde\theta}
			\ar[d]^{\mu\otimes \id } &\Sym ^{\bullet} L\otimes _{\cO_X} \cE	\ar[d]^{\tilde\theta} \\
			\Sym ^{\bullet} L\otimes _{\cO_X} \cE\ar[r]^-{\tilde\theta }&\cE,\\
}
		$$
where $\mu$ is the multiplication in $\Sym ^{\bullet} L$, is commutative.
	\end{enumerate}
\end{Lemma}

\begin{proof}
	Let $x,y$ be local sections of $L$ and $e$ a local section of $\cE$.
	Then the first conditions means that
	$$\theta (x\otimes \theta (y\otimes e))=\theta ( y\otimes \theta (x\otimes e)),$$
	which can be rewritten as $\bar\theta (x) \bar\theta (y)= \bar\theta (y)\bar\theta (x)$, i.e.,
	$[\bar\theta (x) ,\bar\theta (y)]=0=\bar \theta ([x,y])$. This shows equivalence of the first two conditions.
	
	If these conditions are satisfied then there exists a homomorphism of sheaves of $\cO_X$-algebras
	$\Sym ^{\bullet}L\to \End _{\cO_X} (\cE)$ extending $\alpha (\theta)$. This follows from the definition of $\Sym ^{\bullet}L$ as the quotient of the tensor algebra of $L$ by the two-sided ideal generated by local sections of the form $x\otimes y-y\otimes x$.
	
	This map provides $\cE$ with  a $\Sym ^{\bullet}L$-module structure. Clearly, if we have such a structure then also the second condition is satisfied.
\end{proof}

\medskip

Interpretation of a Higgs field as an $\cO_X$-linear map $\theta : L\otimes _{\cO_X} \cE\to \cE$ is sometimes more convenient. For example, we can use it to introduce the following definition that will play an important role in the paper.

\begin{Definition}
	A \emph{system of $L$-Hodge sheaves} is an $L$-Higgs
	sheaf $(\cE, \theta : L\otimes _{\cO_X} \cE \to \cE )$ for which $\cE$ splits into a direct sum $\bigoplus
	\cE^i$ so that $\theta$ maps $L\otimes _{\cO_X} \cE^i$ to $\cE^{i-1}$.
\end{Definition}

\subsection{Morphisms of generalized Higgs sheaves}

Let $X$ be a scheme over a field $k$ and let $L$ be a quasi-coherent $\cO_X$-module. Lemma \ref{equivalent-Higgs}
shows that we can treat an $L$-Higgs sheaf as a pair $(\cE , \theta)$, where $\theta : L\otimes _{\cO_X} \cE\to \cE$ is an $\cO_X$-linear map satisfying certain additional conditions (e.g., condition 1 from Lemma \ref{equivalent-Higgs}).
This point of view is convenient when one wants to consider morphisms between $L$-Higgs sheaves, because 
giving a {morphism of $L$-Higgs sheaves} $\varphi : (\cE _1, \theta _1)\to  (\cE _2, \theta _2) $ is 
equivalent to giving an $\cO_X$-linear map $\varphi: \cE_1\to \cE_2$ such that  the diagram
	$$\xymatrix{L\otimes _{\cO_X} \cE_1\ar[r]^-{\theta_1}\ar[d]^{\id \otimes \varphi}& \cE_1\ar[d]^{\varphi}\\
	L\otimes _{\cO_X} \cE_2\ar[r]^-{\theta_2}&\cE_2\\
}$$
is commutative.
In the following we denote the category of $L$-Higgs sheaves on $X$ by $\Hig _L(\cO_X)$.
By $\Hig^{ref}_L (\cO_X)$ we denote the category of reflexive  $L$-Higgs sheaves on $X$.

\subsection{Reflexive Higgs sheaves}

In this subsection we assume that $X$ is  integral locally Noetherian and  $L$ is a  coherent $\cO_X$-module.
We set $\Omega_L^{[m]}=({\bigwedge}^mL)^*$ for $m\ge 1$. So in particular we have $\Omega_L^{[1]}:=L^*$. We also fix a reflexive coherent $\cO_X$-module $\cE$. 
  
By \cite[Lemma 0AY4]{SP} the sheaf $\cHom_{\cO_X}(L, \cE)$ is also reflexive. In particular, if $X$ is  normal  
then $\cHom_{\cO_X}(L, \cE)\simeq \cE\hat{\otimes} \Omega_L^{[1]}$. Since
$$\Hom _{\cO_X} (L\otimes _{\cO_X} \cE, \cE)=\Hom _{\cO_X} (\cE\otimes _{\cO_X} L, \cE)\mathop{\longrightarrow}^{\simeq}\Hom _{\cO_X}
(\cE, \cHom_{\cO_X}(L, \cE)),$$
we have a canonical isomorphism
$$\Hom _{\cO_X} (L\otimes _{\cO_X} \cE, \cE)\mathop{\longrightarrow}^{\simeq}\Hom _{\cO_X}
(\cE, \cE\hat{\otimes} \Omega_L^{[1]}).$$

\begin{Lemma}\label{comparison-of-dual-definitions}
Assume that $X$ is normal.
If $\theta : L\otimes _{\cO_X} \cE\to \cE$ and $\bar\theta: \cE\to \cE\hat{\otimes} \Omega_L^{[1]}$ are $\cO_X$-linear maps corresponding to each other under the above isomorphism then
	$(\cE, \theta)$ is an $L$-Higgs sheaf if and only if the composition
	$$\xymatrixcolsep{3pc}\xymatrix{
		\cE\ar[r]^-{\bar\theta}& \cE\hat{\otimes} \Omega_L^{[1]}\ar[r]^-{\bar\theta\hat{\otimes }\id}&
		\cE\hat{\otimes} \Omega_L^{[1]}\hat{\otimes} \Omega_L^{[1]}\ar[r]^-{\id \hat{\otimes}\wedge}&
		\cE\hat{\otimes} \Omega_L^{[2]}.\\
	}
	$$
	vanishes. 
\end{Lemma}

\begin{proof}
Note that we have a  canonical isomorphism
$$\Hom _{\cO_X} ({\bigwedge}^2L\otimes _{\cO_X} \cE, \cE)\mathop{\longrightarrow}^{\simeq}\Hom _{\cO_X}
(\cE, \cE\hat{\otimes} \Omega_L^{[2]}).$$
So the required assertion follows from Lemma \ref{equivalent-Higgs} and the fact that the above defined composition $\cE\to	\cE\hat{\otimes} \Omega_L^{[2]}$ corresponds to the composition
	$$\xymatrixcolsep{3pc}\xymatrix{
	{\bigwedge}^2L\otimes _{\cO_X}	\cE\ar[r]^-{\iota \otimes \id }& L\otimes _{\cO_X}L\otimes _{\cO_X} \cE \ar[r]^-{\id{\otimes }\theta}&
	L\otimes _{\cO_X}\cE\ar[r]^-{\theta }&
	\cE .\\} $$
\end{proof}

We need to change the sign in the last map to make it compatible with de Rham sequences for modules over Lie algebroids.

\medskip

The above lemma and Lemma \ref{restriction-to-open}  give a different proof of the following corollary (cf. Subsection \ref{extensions-of-L-modules}).

\begin{Corollary}
Assume that $X$ is normal and  $L$ is reflexive. 
If $j: U\hookrightarrow X $ is a big open subscheme $X$ then
 $j_*$ and $j^*$ define adjoint equivalences of categories $\Hig^{ref} _L(\cO_X)$ and $\Hig^{ref}_{L_U}(\cO_U)$. 
\end{Corollary}

\medskip

\begin{Corollary}\label{passing-to-reflexivization}
If $X$ is normal  and $\theta :  L\otimes _{\cO_X} \cE\to \cE$ is an  $L$-Higgs field on $\cE$
 then we have a canonical structure of an $L^{**}$-Higgs sheaf on $\cE$.
\end{Corollary}

\begin{proof}
	The assertion follows immediately from the previous lemma and the remark that
$\Omega_{L^{**}}^{[1]}= \Omega_L^{[1]}$ and $\Omega_{L^{**}}^{[2]}= \Omega_L^{[2]}$.
\end{proof}

\medskip

Lemma \ref{comparison-of-dual-definitions}
allows us to compare our definition of a Higgs sheaf to that provided in \cite[Definition 5.1]{GKPT2}.

\begin{Corollary}\label{comparison-to-GKPT}
Let $\cE$ be a reflexive coherent $\cO_X$-module on a normal $k$-variety $X$.
Let  $\bar\theta: \cE\to \cE{\otimes} \Omega_X^{[1]}$ be an $\cO_X$-linear map such that 
the composition
$$\xymatrixcolsep{3pc}\xymatrix{
	\cE\ar[r]^-{\bar\theta}& \cE{\otimes} \Omega_X^{[1]}\ar[r]^-{\bar\theta{\otimes }\id}&
	\cE{\otimes} \Omega_X^{[1]}{\otimes} \Omega_X^{[1]}\ar[r]^-{\id {\otimes}[\wedge]}&
	\cE{\otimes} \Omega_X^{[2]}\\
}
$$
vanishes. Then after composing $\bar\theta$ with the reflexivization map $\cE{\otimes} \Omega_X^{[1]}\to \cE \hat{\otimes} \Omega_X^{[1]}$ we can consider the corresponding $\cO_X$-linear map 
$\theta: T_{X/k}\otimes _{\cO_X} \cE\to \cE$. Then $(\cE, \theta)$ is a Higgs sheaf.
\end{Corollary}

\begin{proof}
Let $U\subset X$ be the maximal open subset of the regular locus of $X$ on which $\cE$ is locally free. Note that this open subset is big. By assumption the composition
$$\xymatrixcolsep{3pc}\xymatrix{
	\cE\ar[r]^-{\bar\theta}& \cE{\otimes} \Omega_X^{[1]}\ar[r]^-{\bar\theta{\otimes }\id}&
	\cE{\otimes} \Omega_X^{[1]}{\otimes} \Omega_X^{[1]}\ar[r]^-{\id {\otimes}[\wedge]}&
	\cE{\otimes} \Omega_X^{[2]}\\
}
$$
vanishes on $U$. Taking the pushforward of the restriction of this sequence under the open embedding
$j: U\hookrightarrow X$ gives the sequence from the previous lemma.
\end{proof}

\medskip

\begin{Remark}
Note that not all reflexive Higgs sheaves in our sense come from reflexive Higgs sheaves as defined in \cite{GKPT2}. More precisely, if $(\cE, \theta)$ is a reflexive  Higgs sheaf then we get the corresponding map $\bar\theta: \cE\to \cE\hat{\otimes} \Omega_X^{[1]}$. However, this map does not need to factor through  $\cE\to \cE{\otimes} \Omega_X^{[1]}$. Even if the above $\bar\theta$ factors through
$\cE\to \cE{\otimes} \Omega_X^{[1]}$ then the composition 
$$\xymatrixcolsep{3pc}\xymatrix{
	\cE\ar[r]^-{\bar\theta}& \cE{\otimes} \Omega_X^{[1]}\ar[r]^-{\bar\theta{\otimes }\id}&
	\cE{\otimes} \Omega_X^{[1]}{\otimes} \Omega_X^{[1]}\ar[r]^-{\id {\otimes}[\wedge]}&
	\cE{\otimes} \Omega_X^{[2]}\\
}
$$
vanishes only after composing with the reflexivization $\cE{\otimes} \Omega_X^{[2]}\to \cE \hat{\otimes} \Omega_X^{[2]}$. In this case the induced map $\cE\to {\mathrm{Torsion}}\,(\cE{\otimes} \Omega_X^{[2]})$
can be non-zero, so even in this case we do not get  a Higgs sheaf in the sense of 
\cite[Definition 5.1]{GKPT2}.
\end{Remark}

\subsection{Pullback of generalized Higgs sheaves}\label{pullback-general}

If $f:X\to Y$ is a morphism of schemes and $L$ is a quasi-coherent $\cO_Y$-module then we can easily define the pullback of $L$-Higgs sheaves. Namely, if  $(\cE , \theta : L\otimes _{\cO_Y} \cE\to \cE)$ is an $L$-Higgs sheaf then it is easy to see that
 $$(f^*\cE , f^*\theta : f^*L\otimes _{\cO_X} f^*\cE= f^*(L\otimes _{\cO_Y} \cE)\to f^* \cE)$$
is an $f^*L$-Higgs sheaf (for example one can check condition 1 from Lemma \ref{equivalent-Higgs}). This defines the pullback functor on the corresponding categories of generalized Higgs sheaves, which is functorial with respect to morphisms between schemes.

\subsection{Pullback of Higgs sheaves in characteristic zero}\label{pullback-char-0}

Let $(X, D)$ be a klt pair in the characteristic zero case. Then \cite[Theorems 1.3 and 5.2]{Ke} constructs pullback functor for reflexive differentials on klt pairs that is compatible with the usual pullback of K\"ahler differentials. More precisely, if $f: Y\to X$ is a morphism from a normal variety $Y$ then we get an $\cO_Y$-linear map $d_{ref} f: f^*\Omega_X^{[1]}\to \Omega_Y^{[1]}$. This gives rise to the dual map $T_{Y/k}\to ( f^*\Omega_X^{[1]})^*$. 

If $(\cE,\theta)$ is a Higgs sheaf then the above construction gives a structure of 
$f^{*}T_{X/k}$-Higgs sheaf on $f^*\cE$.  By Corollary \ref{passing-to-reflexivization} this induces a $f^{[*]}T_{X/k}$-Higgs sheaf structure  on $f^*\cE$. Unfortunately, the canonical map $f^{[*]}T_{X/k}\to ( f^*\Omega_X^{[1]})^*$ is not an isomorphism in general and we do not have any canonical map  $T_Y\to f^{[*]}T_{X/k}$. So we cannot pullback general Higgs sheaves on $X$ to Higgs sheaves on $Y$.
However, if  $(\cE,\theta)$  is a  Higgs bundle then we can define its pullback by taking
the composition
$$f^{*}\cE=f^{[*]}\cE\to f^{[*]}(\cE\hat \otimes \Omega_X)=  f^{*}\cE\otimes f^{[*]}\Omega_X\to f^{*}\cE\hat\otimes \Omega_Y^{[1]}.$$
This construction should be compared to \cite[5.3]{GKPT2} and the last sentence in \cite[5.2]{GKPT2}.

\begin{Remark}
 $T_{X/k}$ has a canonical Lie algebroid structure with the standard Lie bracket and identity anchor map.
A \emph{module with an integrable connection} is a $T_{X/k}$-module for the above Lie algebroid structure.
Note that one cannot define reflexive pullback for reflexive modules with an integrable connection even if $Y$ is smooth. The problem is that the pullback would give a reflexive module with an integrable connection.
Such modules are locally free and have vanishing Chern classes. However, one can show explicit examples where the reflexive pullback of a reflexive sheaf underlying a module with an integrable connection does not have vanishing Chern classes. 
\end{Remark}

\subsection{Reflexive pullback for Higgs sheaves under finite morphisms}

Let $f: X\to Y$ be a finite dominant morphism of integral normal schemes, locally of finite type over $k$.

\medskip

\subsubsection{Pullback in the separable case}

Lemma \ref{pull-back-of-dual}
implies that $f^{[*]}T_{Y/k}=(f^*\Omega_{Y/k})^*$. So
we have a canonical map $T_{X/k}\to f^{[*]}T_{Y/k}$ dual to $df: f^*\Omega_{Y/k}\to \Omega_{X/k}.$
Since this map is non-interesting for purely inseparable morphisms, from now on 
we assume that the induced field extension $ k(Y)\hookrightarrow k(X)$ is separable.
In this case the map $T_{X/k}\to f^{[*]}T_{Y/k}$ is injective and it uniquely extends to a homomorphism of sheaves of $\cO_X$-algebras   $\Sym^{\bullet}T_{X/k}\to \Sym^{\bullet} f^{[*]}T_{Y/k}$.

If $(\cE, \theta)$ be a Higgs sheaf on $Y$ then 
$(f^*\cE , f^*\theta)$ is an $f^*T_{X/k}$-Higgs sheaf. Assume that $\cE$ is reflexive.
Taking reflexivization we get  an $f^*T_{X/k}$-Higgs module structure on $f^{[*]}\cE$. 
By Corollary \ref{passing-to-reflexivization} we also have an induced  $f^{[*]}T_{X/k}$-Higgs module
structure on $f^{[*]}\cE$. Then the homomorphism $\Sym^{\bullet}T_{X/k}\to \Sym^{\bullet} f^{[*]}T_{Y/k}$ provides $f^{[*]}\cE$ with a canonical Higgs module structure. This Higgs module will be denoted by 
$f^{[*]}(\cE, \theta)=(f^{[*]}\cE , f^{[*]}\theta)$.

\medskip

One can also describe the above construction explicitly in the following way (this will be useful in the next construction). Namely, let $(\cE, \theta: \cE\to \cE\hat {\otimes} \Omega_Y^{[1]})$ be a reflexive Higgs sheaf on $Y$ (see Lemma \ref{comparison-of-dual-definitions}).
Then there exists a big open subset $V\subset Y$ such that $(\cE, \theta)$ is (log) smooth on $V$.
Since $f$ is finite, $U=f^{-1}(V)$ is a big open subset of $X$. Let $j: U\hookrightarrow X$ be the corresponding open embedding. Then we can define the map
$$(f^*\cE)_U=f^*(\cE_V)\mathop{\longrightarrow}^{f^*\theta}f^*(\cE_V\otimes _{\cO_V} \Omega_V)= f^*(\cE_V)\otimes _{\cO_U} f^*\Omega_V\mathop{\longrightarrow}^{\id \otimes df} (f^*\cE)_U\otimes _{\cO_U} \Omega_U.$$
This gives the map $T_U\otimes (f^*\cE)_U\to (f^*\cE)_U$, which leads to
$$T_X\otimes _{\cO_X} f^{[*]}\cE=j_*(T_U)\otimes _{\cO_X}j_*((f^*\cE)_U)\to j_*(T_U\otimes_{\cO_U} (f^*\cE)_U)\to j_*((f^*\cE)_U)=f^{[*]}\cE.$$
So we get the induced map $f^{[*]}\theta : f^{[*]}\cE\to f^{[*]}\cE\hat{\otimes} _{\cO_{X}}\Omega_X ^{[1]}$.
One can easily check that $ (f^{[*]}\cE, f^{[*]}\theta)$ is  a reflexive Higgs sheaf on $X$.

\subsubsection{Pullback in the inseparable case}

Unfortunately, the above construction is rather useless in case $f$ is purely inseparable as then $df=0$
and $f^{[*]}\theta$ always vanishes. However, if $f=F_X$  and the big open subset $U$ is $F$-liftable
(see Definition \ref{definition-F-liftable}) then we have an induced map $\xi: F^*_U\Omega _U\to \Omega_U $. Now if in the above construction we replace $df$ by $\xi$, then for any reflexive Higgs sheaf  $(\cE, \theta)$ on $X$
we can define 
$$F_X^{[*]}\theta : F_X^{[*]}\cE\to F_X^{[*]}\cE\hat \otimes _{\cO_{X}}\Omega_X ^{[1]}.$$
This construction is used in the proof of Lemma \ref{classes-Cartier-transform}.
Note that this map depends on the choice of the lifting.

\medskip

Similar constructions as above work also for log pairs.

\subsection{Semistability for $L$-modules}

Let $X$ be a normal projective variety of dimension $n$ defined over an algebraically closed field $k$.
Let $(L_1,...,L_{n-1})$ be a collection of nef line bundles on $X$ and let $L$ be a Lie algebroid on $X/k$ such that $L$ is coherent as an $\cO_X$-module.

\begin{Definition}
	Let $(\cE, \nabla)$ be an $L$-module such that $\cE$ is coherent and torsion free as an $\cO_X$-module. We say that $(\cE, \nabla)$ is \emph{slope  $(L_1,...,L_{n-1})$-semistable} 
	if for any $L$-submodule  $(\cF, \nabla _{\cF})\subset (\cE, \nabla)$ we have
	$\mu (\cF)\le \mu (\cE)$. We say that  $(\cE, \nabla)$ is \emph{slope  $(L_1,...,L_{n-1})$-stable} 
	if for any $L$-submodule  $(\cF, \nabla _{\cF})\subset (\cE, \nabla)$ such that $1\le \rk \cF<\rk\cE$ we have	$\mu (\cF)< \mu (\cE)$. 
\end{Definition}

 In further part of this subsection we consider slope semistability with respect to a fixed collection of nef line bundles and we omit it from the notation.
 
If $(\cE, \theta: L\otimes _{\cO_X}\cE \to \cE )$ is an $L$-Higgs sheaf then it is slope semistable
if and only if for every $\cO_X$-submodule $\cF\subset \cE$ such that the image of 
$L\otimes  _{\cO_X}\cF \to L\otimes _{\cO_X}\cE \to \cE$ is contained in $\cF$, we have 
$\mu (\cF)\le \mu (\cE)$ (and similarly for slope stability).
This should be compared with \cite[Definition 4.13]{GKPT2} that considers semistability using 
so called generically $\theta$-invariant subsheaves. It is easy to see that for reflexive Higgs sheaves in the sense of \cite{GKPT2} the obtained notions of semistability coincide.

Let us also remark that if $(\cE, \theta)$ is a system of $L$-Hodge sheaves then we can define notion of semistability using only subsystems of $L$-Hodge sheaves. It is easy to see that this is equivalent to semistability of $(\cE, \theta)$ as an $L$-Higgs sheaf. We will use this fact in Section \ref{Section:Miyaoka-Yau}.

\medskip

We have the following general lemma allowing to bound instability of slope semistable $L$-modules. It is a weak form of \cite[Lemma 5]{La3} but it works also for nef polarizations

\begin{Lemma}\label{L-module-ss-bound}
	Let $A$ be  an ample Cartier divisor $A$ such that $L(A)$ is globally generated.
Let $(\cE, \nabla)$ be an $L$-module such that $\cE$ is coherent and torsion free of rank $r$ as an $\cO_X$-module. If $(\cE, \nabla)$ is slope semistable then
$$\mu _{\max} (\cE)-\mu _{\min} (\cE) \le (r-1) AL_1...L_{n-1} .$$
\end{Lemma}

\begin{proof}
	If $\cF \subset \cE$ is an $\cO_X$-submodule then an $L$-module structure on  $(\cE, \nabla)$
	induces an $\cO_X$-linear map $L\otimes _{\cO_X}\cF\to \cE/\cF$. If this map vanishes then 
	$\cF$ has a natural structure of an $L$-submodule of  $(\cE, \nabla)$.
	
	Let $\cE_0=0\subset \cE_1\subset ...\subset \cE_s=\cE$ be the  Harder--Narasimhan filtration of $\cE$ and let us set $\cE ^{i}:=\cE_i/\cE_{i-1}$ for $i=1,...,s$. Then we have non-zero $\cO_X$-linear maps
	$$L\otimes _{\cO_X}\cE_i\to \cE /\cE_i$$
for $i=1,...,s-1$.  Since for some $N>0$ we have a surjective map $\cO _X ^{\oplus N}(-A)\to L$
 there exists  for every $i=1,...,s$ a non-zero map $\cE_i(-A)\to \cE /\cE_i.$
	So we have $\mu (\cE ^{i})- AL_1...L_{n-1}\le \mu (\cE ^{i+1})$.
	Summing these inequalities we get
	$$\mu _{\max} (\cE)-\mu _{\min} (\cE) \le (s-1) AL_1...L_{n-1}\le (r-1) AL_1...L_{n-1}. $$
\end{proof}

\subsection{Strong restriction theorem  for generalized Higgs sheaves}

We keep the notation from the previous subsection but we assume that $L$ has trivial Lie algebroid structure.  The same proofs as that of Theorem \ref{strong-restriction-stable} and Corollary \ref{Bogomolov-restriction} give the following theorem
(cf. \cite[Theorem 9]{La3} in the smooth case). See Subsection \ref{pullback-general} for the definition of pullback used in the statement. 
 
\begin{Theorem}	\label{restriction-for-operators}
Let  $(\cE , \theta )$ be a reflexive $L$-Higgs sheaf of rank $r\ge 2$. Let us assume that  
$d=L_1^2L_2...L_{n-1}>0$ and let  $m$ be an integer such that
	$$m>\left\lfloor \frac{r-1}{ r}\int_X\Delta (\cE)L_2\dots L_{n-1} +\frac{1}{dr(r-1)}
	+\frac{(r-1)\beta _r}{dr}\right\rfloor .$$
Let $H\in |L_1^{\otimes m}|$  be an irreducible normal divisor and let $i: H\hookrightarrow X$ denote the corresponding embedding.
	\begin{enumerate}
		\item If  $(\cE , \theta )$ is slope $(L_1,...,L_{n-1})$-stable 	then 
		the  $i^*L$-Higgs sheaf $(i^*\cE , i^*\theta )$ is  slope $(i^*L_2,\dots ,i^*L_{n-1})$-stable.
		\item If  $(\cE , \theta )$ is slope $(L_1,...,L_{n-1})$-semistable 
	and restrictions of all quotients of a Jordan--H\"older filtration of $(\cE , \theta )$ to $H$ are torsion free then  	the  $i^*L$-Higgs sheaf $(i^*\cE , i^*\theta )$ is  slope $(i^*L_2,\dots ,i^*L_{n-1})$-semistable. 
	\end{enumerate}
\end{Theorem}

Note that the above theorem should be thought of as a restriction theorem for sheaves with operators and not a genuine restriction theorem for Higgs sheaves (cf. Theorem \ref{restriction-for-Higgs}).

\section{Bogomolov's inequality for logarithmic Higgs sheaves on singular varieties}

This section contains proofs of Theorems  \ref{Main5} and \ref{Main3}. The main idea is to use Ogus--Vologodsky's correspondence and suitably generalized Higgs--de Rham sequences.

\subsection{Ogus--Vologodsky's correspondence on normal varieties}

Let $X$ be a normal variety defined over an algebraically closed field $k$ of positive characteristic $p$.
Let $D$ be an effective reduced Weil divisor on $X$. 

\medskip

Let us consider a (reflexive) Lie algebroid $L$, whose underlying $\cO_X$-module is $T_X(\log D)$ with the
anchor map $T_X (\log D)\hookrightarrow T_{X/k}$ and the Lie bracket induced from $T_{X/k}$.
An $L$-module for this Lie algebroid is called \emph{an $\cO_X$-module with an integrable logarithmic connection } on $(X,D)$. In fact, $L$ carries a restricted Lie algebroid structure 
(see \cite[Section 4]{La2}) given by raising logarithmic derivations to the $p$-th power.
This allows us to talk about \emph{logarithmic $p$-curvature $F_X^*L\to \End _{\cO_X} \cE$}
of an $\cO_X$-module with an integrable logarithmic connection.

If $(\cE , \nabla : T_{X} (\log D)\to \End _k \cE)$ is a reflexive $\cO_X$-module with an integrable logarithmic connection then we can also define its residues in the following way.
For every open subset $U\subset X$ an element $\delta \in  (T_{X} (\log D) )(U)$ can be considered as a logarithmic $k$-derivation of $\cO_U$. Let $J$ be the ideal subsheaf of $\cO_U$ generated by the image of $\delta$. Then 
$$ \cE |_U\mathop{\longrightarrow}^{\nabla (\delta)} \cE |_U\to (\cE |_U)/( J\cE |_U)$$  
induces an endomorphism $\rho _{\delta} $ of $(\cE |_U)/( J\cE |_U)$ called the residue associated to $\delta$. 
We say that the residues of $(\cE , \nabla)$ are \emph{nilpotent of order $\le p$} if for every $U\subset X$ and $\delta  \in  (T_{X} (\log D) )(U)$ we have $\rho _{\delta}^p=0$.

\medskip

Similarly, one can consider $T_X(\log D)$ with the trivial Lie bracket and zero anchor map. 
Modules over this reflexive Lie algebroid are called \emph{logarithmic Higgs sheaves} on $(X,D)$.
We say that a logarithmic Higgs sheaf $(\cE, \theta : T_X(\log D)\to \End_{\cO_X} \cE)$ has \emph{ a nilpotent Higgs field of level $\le (p-1)$} if for every open subset $U\subset X$ and every 
$\delta  \in  (T_{X} (\log D) )(U)$ we have $\theta (\delta) ^p=0$, where $\theta (\delta)$ is considered as an $\cO_U$-linear endomorphism of $\cE |_U$.
 
\medskip

The following theorem generalizes Ogus--Vologodsky's correspondence to normal varieties:

\begin{Theorem}\label{log-smooth-OV-correspondence}
  Let us assume that there exists a big open subset $U\subset X$ such that the pair
$(U, D_U=D\cap U)$ is log smooth and liftable to $W_2(k)$. Let us fix
 a lifting $(\tilde U, \tilde D_U)$ of $(U, D_U)$. Then there exists a Cartier
  transform $C_{(\tilde U,\tilde D_U)}$, which defines an equivalence of
  categories of reflexive $\cO_X$-modules with an integrable
  logarithmic connection whose logarithmic $p$-curvature is nilpotent
  of level less or equal to $p-1$ and the residues are nilpotent of
  order less than or equal to $p$ on $U$, and reflexive logarithmic Higgs
  $\cO_X$-modules with a nilpotent Higgs field of level less or equal
  to $p-1$.
\end{Theorem}

\begin{proof}
As remarked in Subsection \ref{extensions-of-L-modules}, for any reflexive Lie algebroid $L$ and any big open subset $U\subset X$, we have an equivalence of categories of reflexive $L$-modules on $X$  and reflexive $L_U$-modules on $U$.
So  the results of Ogus and Vologodsky in the usual case (see \cite{OV}) and Schepler in the logarithmic one (see \cite{Sc}; see also \cite[Theorem 2.5]{La3} and \cite[Appendix]{LSYZ}) 
give the above correspondence on $U$. One needs only to check that extension to $X$ preserves the remaining conditions. For Higgs modules it is clear that having a nilpotent Higgs field of level $\le (p-1)$ on $U$ gives the same condition on $X$. Similarly, for modules with logarithmic connections checking nilpotency of the logarithmic $p$-curvature on $U$ implies the one on $X$.
\end{proof}

\medskip

A quasi-inverse to  $C_{(\tilde U,\tilde D_U)}$ is denoted by $C_{(\tilde U,\tilde D_U)}^{[-1]} $  (or simply $C^{[-1]}$) and it is called the \emph{reflexivized inverse Cartier transform}.

\subsection{Strong restriction theorem  for logarithmic Higgs sheaves}

We keep the notation from the previous subsection. 

\begin{Definition}
	Let  $j:H\hookrightarrow X$ be a locally principal closed subscheme of $X$ (i.e., a scheme associated to 
	an effective Cartier divisor). We say that $H$ is \emph{good} for the pair $(X, D)$ if the following conditions are satisfied:
	\begin{enumerate}
		\item $H$ is irreducible and normal, 
		\item  $H$ is not contained in any irreducible component of $D$,
		\item  If $U\subset X$ is the maximal open subset on which $(U, D\cap U)$ is log smooth then
		$H\cap U$ is big in $H$, 
		\item The pair $(H\cap U, D_H\cap U)$, where $D_H=H\cap D$, is log smooth.
	\end{enumerate}
\end{Definition}

\medskip

If $H$ is good for $(X, D)$ then we have a canonical map
$$T_H(\log D_H)\to (j^*T_X(\log D))^{**}$$
obtained by extension of the canonical map $T_{H\cap U}(\log D_H\cap U)\to j^*T_X(\log D)|_{H\cap U}$.
In particular, if $H$ is good for $(X, D)$ then Corollary \ref{passing-to-reflexivization} shows that any
logarithmic Higgs sheaf $(\cE , \theta )$  on $(X,D)$ gives rise to a reflexive logarithmic Higgs sheaf structure 
$$j^{[*]}(\cE , \theta ):= (j^{[*]} \cE, T_H(\log D_H)\otimes _{\cO_H} j^{[*]} \cE \to (j^*T_X(\log D))^{**}\otimes _{\cO_H} j^{[*]} \cE\to j^{[*]} \cE) $$
on $j^{[*]} \cE:= (j^*\cE )^{**}$ over $(H, D_H)$. 

\medskip

\begin{Remark}
	If $L$ is a very ample line bundle then  Bertini's theorem implies that for all $m\ge 1$ a general hypersurface $H\in |L^{\otimes m}|$ is good for $(X,D)$.
\end{Remark}

\medskip

\begin{Theorem}	\label{restriction-for-Higgs}
	Let  $(\cE , \theta )$ be a reflexive logarithmic Higgs sheaf of rank $r\ge 2$ on $(X,D)$. 
	Let us assume that $L_1$ is ample and let $m_0$ be a non-negative integer such that $T_X (\log \, D)\otimes L_1^{\otimes m_0}$ is globally generated. Assume also that  $d=L_1^2L_2...L_{n-1}>0$ and let   $m$ be an integer such that
	$$m>\max \left(\left\lfloor \frac{r-1}{ r}\int_X\Delta (\cE)L_2\dots L_{n-1} +\frac{1}{dr(r-1)}
	+\frac{(r-1)\beta _r}{dr}\right\rfloor , 2(r-1)m_0^2 \right).$$
	Let $H\in |L_1^{\otimes m}|$ be good for $(X,D)$ with closed embedding $j:H\hookrightarrow X$.
	\begin{enumerate}
		\item If  $(\cE , \theta )$ is slope $(L_1,...,L_{n-1})$-stable then  $j^{[*]}(\cE , \theta )$ on $(H,D_H)$ is  slope $(j^*L_2,\dots ,j^*L_{n-1})$-stable.
		\item If  $(\cE , \theta )$ is slope $(L_1,...,L_{n-1})$-semistable and restrictions of all quotients of a Jordan--H\"older filtration of $(\cE , \theta )$ to $H$ are torsion free
		then  $j^{[*]}(\cE , \theta )$ is  slope $(j^*L_2,\dots ,j^*L_{n-1})$-semistable.  
	\end{enumerate}
\end{Theorem}

\begin{proof}
Using Theorem \ref{restriction-for-operators}, one can follow the proof of
\cite[Theorem 10]{La3} to obtain the first part of the theorem.

Now let us remark that if  $0= (\cE_0, \theta_0)\subset (\cE_1, \theta_1)\subset ...\subset (\cE_s, \theta_s)=(\cE , \theta )$ is a  Jordan--H\"older filtration of $\cE$ and $(\cF_i, \tilde \theta_{i}):=((\cE_i, \theta_i)/(\cE_{i-1}, \theta_{i-1}))^{**}$ then by Lemma \ref{needed-for-boundedness} we have 
\begin{align*}
		\frac{\int _X \Delta  (\cE)L_2...L_{n-1}}{r}\ge \sum _i \frac{\int _X \Delta  (\cF_i)L_2...L_{n-1}}{r_i},
\end{align*}
where $r_i=\rk \cF_i$. So by the first part all $j^{[*]}(\cF_i, \tilde \theta_{i})$ are slope $(j^*L_2,\dots ,j^*L_{n-1})$-semistable. Since $H$ is good for $(X,D)$, there exists a big open subset $U\subset X$ such that $(U, D\cap U)$ is log smooth and $H\cap U$ is big in $H$.  Now a logarithmic Higgs subsheaf destabilizing
$j^{[*]}(\cE , \theta )$ would destabilize it on $H\cap U$. This would show that one of the restrictions
$j^*(\cE_i, \theta_i)|_{H\cap U}$ is not slope $(j^*L_2,\dots ,j^*L_{n-1})$-semistable. But this contradicts the fact that the reflexivization of its extension to $H$ (which is equal to $j^{[*]}(\cF_i, \tilde \theta_{i})$)
is slope $(j^*L_2,\dots ,j^*L_{n-1})$-stable.
\end{proof}

\medskip

For $r=2$ the assumptions of this theorem can be slightly relaxed (cf. \cite[Theorem 10]{La3}).
Note that unlike in \cite{La3} we do not have any assumptions on lifting on $X$. These assumptions were added in \cite{La3} only to avoid the term containing $\beta _r$ so that the results could hold uniformly in all characteristics (including $0$). The above result is restricted to the positive characteristic and
it was not known in the characteristic zero case even if one assumes that $D=0$ and $X$ has klt singularities (cf. \cite[Theorem 5.22]{GKPT2} for a non-effective restriction theorem for general hypersurfaces).  The above theorem will be used to obtain a strong restriction theorem for Higgs sheaves in characteristic zero in Section \ref{Section:characteristic-0}.

\subsection{Deformations to systems of Hodge sheaves}

Let $X$ be a normal projective variety defined over an algebraically closed field $k$
and let $L$ be a Lie algebroid on $X$, which is coherent as an $\cO_X$-module.

 It is convenient to consider $L$-modules  as modules over the universal enveloping algebra $\Lambda _L$
 of differential operators associated to $L$ (see \cite[Section 2.2]{La2}).
 So we consider an $L$-module as a pair  $(\cE, \nabla)$, where $\cE$ is a quasi-coherent $\cO_X$-module
and $\nabla: \Lambda _L\otimes_{ \cO_X} \cE \to \cE$ is a $\Lambda _L$-module structure. 
If the underlying sheaf of an $L$-module is coherent as
an $\cO_X$-module, we say (at the risk of  abusing the notation) that $(\cE, \nabla)$ is a \emph{coherent $L$-module}. 
If the underlying sheaf of an $L$-module is coherent and torsion free as
an $\cO_X$-module, we say (again abusing the notation) that $(\cE, \nabla)$ is a \emph{torsion free $L$-module}. 
 
 If  $(\cE, \nabla)$ is a coherent $L$-module then we say that a filtration $\cE=N^0\supset
N^1\supset ...\supset N^m=0$ satisfies \emph{Griffiths
transversality} if it is a filtration of $\cE$ by coherent $\cO_X$-submodules and
$\nabla (\Lambda _L\otimes_{ \cO_X} N^i)\subset N^{i-1}$. For every such filtration the associated graded object $\Gr
_N(\cE):=\bigoplus _i N^i/N^{i+1}$ carries a canonical coherent
$L$-Higgs module structure $\theta: L\otimes_{ \cO_X}\Gr
_N(\cE)\to \Gr _N(\cE)$ defined by $\nabla$.
This can be seen by considering the following commutative diagram:
$$
\xymatrix{
\Lambda_0\otimes_{ \cO_X} N^{i+1}\ar[r]\ar[d]&\Lambda_0\otimes_{ \cO_X} N^{i}\ar[r]\ar[d]&N^i\ar[d]\\
\Lambda_1\otimes_{ \cO_X} N^{i+1}\ar[r]\ar[d]\ar[rru]&\Lambda_1\otimes_{ \cO_X} N^{i}\ar[r]\ar[d]&N^{i-1}\ar[d]\\
\Lambda_1/\Lambda_0\otimes_{ \cO_X} N^{i+1}\ar[r]\ar[d]&\Lambda_1/\Lambda_0\otimes_{ \cO_X} N^{i}\ar[r]\ar[d]&N^{i-1}/N^i\ar[d]\\
0&0&0\\
}
$$
where $\Lambda_0\subset \Lambda _1\subset ... \subset \Lambda _L$ is the standard filtration on $\Lambda _L$.
It follows from the above diagram that the map $\Lambda_1\otimes_{ \cO_X} N^{i+1}\to N^{i-1}/N^i$ is zero and hence
the map $\Lambda_1/\Lambda_0\otimes_{ \cO_X} N^{i+1}\to N^{i-1}/N^i$ is also zero. So we have an induced map $\Lambda_1/\Lambda_0\otimes_{ \cO_X} N^{i}/N^{i+1}\to N^{i-1}/N^i$. But $\Lambda_1/\Lambda_0=L$ and one can easily check that the obtained map gives an $L$-Higgs module structure on  $\Gr_N(\cE)$.
 Note also that by construction the obtained pair $(\Gr_N(\cE), \theta)$ is a system of $L$-Hodge sheaves on $X$. 

\medskip

In the remainder of this section to define semistability we fix a collection
 $(L_1,...,L_{n-1})$ of nef line bundles on $X$ such that $L_1L_2....L_{n-1}$ is numerically non-trivial.

We say that a Griffiths transverse
filtration $N^{\bullet}$ on $(\cE, \nabla)$ is \emph{slope
gr-semistable} if the associated $L$-Higgs sheaf
$(\Gr_N(\cE), \theta)$ is (torsion free and) slope semi\-sta\-ble. A \emph{partial
$L$-oper} is a triple $(\cE, \nabla, N^{\bullet})$ consisting of a
torsion free coherent $\cO _X$-module $\cE$ with a
${\Lambda_L}$-module structure $\nabla$ and a Griffiths transverse
filtration $N^{\bullet}$, which is slope gr-semistable.

\medskip

\begin{Remark}
Note that analogous definitions in \cite[Section 5.2]{La2} work only for smooth Lie algebroids. 
The above definitions allow us to deal with general Lie algebroids and  they are equivalent to those in 
 \cite[Section 5.2]{La2} in case of smooth Lie algebroids. 
\end{Remark}

\medskip

The following theorem can be proven in the same way as \cite[Theorem 5.5]{La4}.
The only difference is that in the proof one needs to consider $L$-modules as 
$\Lambda_L$-modules.

\begin{Theorem} \label{existence-of-gr-ss-Griffiths-transverse-filtration}
If $(\cE, \nabla)$ is slope semistable then there exists a
canonically defined slope gr-semistable Griffiths transverse
filtration $S^{\bullet}$ on $(\cE, \nabla)$ providing it with a
partial $L$-oper structure. This filtration is preserved by the
automorphisms of $(\cE,\nabla)$.
\end{Theorem}

The above filtration $S^{\bullet}$ is called \emph{Simpson's filtration}.
Even in the case of a trivial Lie algebroid structure on $L$ the above theorem gives 
a non-trivial corollary:

\begin{Corollary}\label{deformation-to-system}
  Let $(\cE, \theta : L \otimes _{\cO_X} \cE\to \cE)$ be a slope
  semistable  $L$-Higgs sheaf. Then there exists a decreasing
  filtration $\cE=N^0\supset N^1\supset ...\supset N^m=0$ such that
  $\theta (L \otimes _{\cO_X}N^i)\subset N^{i-1}$ and the associated
  graded is a slope semistable system of $L$-Hodge sheaves.
  \end{Corollary}

\subsection{Higgs--de Rham sequences on normal varieties}

Let $X$ be a normal projective variety defined over an algebraically closed field $k$ of positive characteristic $p$.
Let $D$ be an effective reduced Weil divisor on $X$.

 Let us assume that $(X, D)$ is almost liftable to $W_2(k)$. Then we can
 find a big open subset $U\subset X$ such that the pair
$(U, D_U=D\cap U)$ is log smooth and liftable to $W_2(k)$. Let us fix
 a lifting $(\tilde U, \tilde D_U)$ of $(U, D_U)$. 

\medskip

Let $(\cE, \theta : T_X (\log D)\otimes \cE\to \cE)$ be a  reflexive logarithmic Higgs
 $\cO_X$-module of rank $r\le p$. Let us assume that $(\cE, \theta)$ is slope semistable. Then
 by Corollary \ref{deformation-to-system} there exists a canonical filtration
 $N^{\bullet}$ on $\cE$ such that the associated graded $(\bar \cE_0, \bar \theta _0) $
  is a slope semistable system of logarithmic Hodge sheaves. 
  Let $(\cE_0, \theta_0)$ be the reflexive hull of $(\bar \cE_0, \bar \theta _0) $. By construction, 
 it is a slope semistable reflexive logarithmic system of Hodge sheaves. In particular, since its rank $r$ is $\le p$,
 it is also  a reflexive logarithmic Higgs $\cO_X$-module with a nilpotent Higgs field of level less or equal
 to $p-1$. So we can define 
$(V_{0}, \nabla _0):= C_{(\tilde U,\tilde D_U)}^{[-1]} (\cE_0, \theta _0) $.  
Let $S_0^{\bullet}$  be (decreasing) Simpson's filtration on $(V _0, \nabla _0) $ and let 
$(\bar \cE_1=\Gr_{S_0}(V _0),\bar \theta _1 ) $ be the associated system of Hodge sheaves. Then we set $(\cE_1,\theta_1):=((\bar \cE_1)^{**},
{\bar \theta_1}^{**} )$ and repeat the procedure. In this way we get the following sequence
$${  \xymatrix{
		(\cE, \theta) \ar[rd]^{(\Gr _{N})^{**}}	&& (V_0, \nabla _0)\ar[rd]^{(\Gr _{S_0})^{**}}&& (V_1, \nabla _1)\ar[rd]^{(\Gr _{S_1})^{**}}&\\
	&	(\cE_0, \theta _0)\ar[ru]^{C^{[-1]}}&&	(\cE_1, \theta_1)\ar[ru]^{C^{[-1]}}&&...\\
}}$$
in which  each logarithmic Higgs sheaf  $(\cE_j, \theta_j)$ is reflexive rank $r\le p$  and slope semistable.
We call this sequence a \emph{canonical Higgs--de Rham sequence} of $(\cE, \theta)$.

\medskip

\begin{Remark}
Higgs de Rham sequences were invented by G. Lan. M. Sheng and K. Zuo in \cite{LSZ} and their existence was proven in \cite{LSZ} and \cite{La2}. Canonical Higgs--de Rham sequences in the above sense first appeared in the proof of \cite[Lemma 3.10]{La5}. They are better suited to dealing with normal varieties as one cannot define suitable Chern classes for torsion free sheaves on normal varieties.
\end{Remark}

\begin{Remark}
Although the above construction is very general, it does not seem easy to compare numerical invariants of 
the sheaves $\cE_i$ without some further assumptions on the singularities of the pair $(X,D)$.
\end{Remark}

\subsection{Inverse Cartier transform on log varieties with  locally $F$-liftable singularities.}

Let $X$ be a normal variety defined over an algebraically closed field $k$ of positive characteristic $p$.
We define the \emph{Grothendieck group $K^{\refl }(X)$  of reflexive sheaves on $X$} as 
the free abelian group on the isomorphism classes $[\cE]$ of coherent reflexive $\cO_X$-modules modulo
the relations $[\cE_2]=[\cE_1]+[\cE_3]$ for each locally split short exact sequence
$$0\to \cE_1\to \cE_2\to \cE_3\to 0$$
of coherent reflexive $\cO_X$-modules.

For a coherent $\cO_X$-module $\cF$ we denote by
$\nabla_{can}^{\cF}$ the canonical connection on $F_X^*\cF$ given by differentiating along the fibers of the Frobenius morphism.

\begin{Lemma}\label{classes-Cartier-transform}
Let $D$ be an effective reduced Weil divisor on $X$ such that $(X, D)$ is liftable to $W_2(k)$ and it is locally $F$-liftable.
If $(\cE, \theta)= (\bigoplus \cE^i, \theta)$ is a reflexive system of  logarithmic Hodge sheaves on $(X, D)$
and we set $\cE_j =\bigoplus _{j\le i}\cE^i$ with induced $\theta _j$ then for every $j$ we have  a short exact sequence
$$0\to C^{[-1]} (\cE _j, \theta _j)\to C^{[-1]} (\cE _{j+1}, \theta _{j+1})\to (F_X^{[*]}\cE^{j+1},\nabla_{\can}^{\cE_{j+1}})\to 0$$
of reflexive $\cO_X$-modules with a logarithmic connection, 
which is locally split as a sequence of $\cO_X$-modules. In particular, we have
$[C^{[-1]}\cE]=[F_X^{[*]}\cE]$
in $K^{\refl }(X)$.
\end{Lemma}

\begin{proof}
By  construction we have a short exact sequence of Higgs sheaves
$$0\to (\cE _j, \theta _j)\to (\cE _{j+1}, \theta _{j+1})\to (\cE^{j+1},0)\to 0,$$
which is split as a sequence of $\cO_X$-modules.
Applying $ C^{[-1]}$ to this sequence we get
	$$0\to C^{[-1]} (\cE _j, \theta _j)\to C^{[-1]} (\cE _{j+1}, \theta _{j+1})\to (F_X^{[*]}\cE^{j+1},\nabla_{\can}^{\cE_{j+1}})\to 0,$$
	because $ C^{[-1]}  (\cE^{j+1},0)=(F_X^{[*]}\cE^{j+1},\nabla_{\can}^{\cE_{j+1}})$. So it is sufficient to show that this sequence is locally split. 	
To do so we fix a point $x\in X$ and an open neighbourhood $x\in U\subset X$, which is $F$-liftable. 
Let $V$ be a big open subset of $U$, which is contained in the log smooth locus of $(X,D)$.
The pair $(V, D\cap V)$ has an $F$-lifting $\tilde F_V: \tilde V\to \tilde V$ compatible with the $W_2(k)$-lifting $(\tilde V, \tilde D)$ induced from the given $W_2(k)$-lifting of  $(X, D)$.
On $V$ we have a  short exact sequence of modules with integrable connections
$$0\to (F_V^*\cE _j, \nabla_{\can}^{\cE_j})\to (F_V^*\cE _{j+1}, \nabla_{\can}^{\cE_{j+1}})\to (F_V^*\cE^{j+1},\nabla_{\can}^{\cE^{j+1}})\to 0,$$
which is split as a sequence of $\cO_V$-modules.
Extending the above sequence  to $U$, we get a  short exact sequence of reflexive $T_U$-modules
$$0\to F_U^{[*]}\cE _j\to F_U^{[*]}\cE _{j+1}\to  F_U^{[*]}\cE^{j+1}\to 0,$$
which is split as a sequence of $\cO_U$-modules.
By construction (see Section \ref{Appendix}) $C^{[-1]} (\cE _j)|_U\simeq F_U^{[*]}\cE _{j}$ and  $C^{[-1]} (\theta _j)|_U$ is obtained by extension of the logarithmic connection $\nabla _{\can}^{\cE _j}+ \zeta_V (F^*_V\theta _j)$, where 
$\zeta_V: =p^{-1}\tilde F_V: F_V^*\Omega _V\to \Omega _V$ (see Subsection \ref{Deligne-Illusie}).
Since the above isomorphisms are compatible with restrictions to $V$, we see that the sequence 
$$0\to C^{[-1]} (\cE _j, \theta _j)\to C^{[-1]} (\cE _{j+1}, \theta _{j+1})\to (F_X^{[*]}\cE^{j+1},\nabla_{\can}^{\cE_{j+1}})\to 0$$
is split as a sequence of $\cO_U$-modules.
\end{proof}

\begin{Corollary} \label{Chern-classes-Cartier-transform}
Let $X$ be a normal projective variety with a collection $(L_1,...,L_{n-2})$ of nef line bundles.
Let $D$ be an effective reduced Weil divisor on $X$ such that $(X, D)$ is almost liftable 
to $W_2(k)$ and it has  $F$-liftable singularities in codimension $2$.
Then we have 
$$\int _X \ch _2  (C^{[-1]}\cE)L_1...L_{n-2}=p^2\int _X\ch _2   (\cE)L_1...L_{n-2}.$$
\end{Corollary}

\begin{proof}
By Theorem \ref{properties-of-ch_2} we can reduce the assertion to the surface case. Then Theorem \ref{equivalences-for-almost-liftings} says that $(X,D)$ satisfies assumptions of 
Lemma \ref{short-exact-seq} and hence we get
\begin{align*}
\int _X \ch _2  (C^{[-1]}\cE)L_1...L_{n-2}&=\sum _j \int _X \ch _2  (F_X^{[*]}  \cE ^j)L_1...L_{n-2}=p^2 \sum _j \int _X \ch _2  ( \cE ^j)L_1...L_{n-2}\\
&=p^2\int _X\ch _2   (\cE)L_1...L_{n-2}.
\end{align*}
\end{proof}

\subsection{Bogomolov's inequality for Higgs sheaves}

In this subsection we give the first version of Bogomolov's inequality for 
logarithmic Higgs sheaves on singular varieties. 
The following theorem generalizes Bogomolov's inequality for logarithmic Higgs
sheaves to singular varieties (see \cite[Theorem 8]{La3} in case $X$ is smooth
and \cite[Theorem 3.3]{La4} for the log smooth case).

\begin{Theorem} \label{log-Bogomolov-with-lifting} 
	Let  $(L_1,...,L_{n-1})$ be a collection of nef line bundles on $X$ such that $L_1L_2....L_{n-1}$ is numerically non-trivial. Assume that the pair $(X,D)$ is  almost liftable to $W_2(k)$ and and it has  $F$-liftable singularities in codimension $2$.  Then for any
	slope $(L_1,...,L_{n-1})$-semistable logarithmic reflexive Higgs sheaf $(\cE, \theta)$ of rank
	$r\le p$ we have
	$$\int_X \Delta (\cE) L_2...L_{n-1}\ge 0.$$
\end{Theorem}

\begin{proof}
Let $(\cE, \theta : T_X (\log D)\otimes _{\cO_X}\cE\to \cE)$ be a  reflexive logarithmic Higgs
 $\cO_X$-module of rank $r\le p$. Let us assume that $(\cE, \theta)$ is slope semistable. 
Let 
$${  \xymatrix{
		(\cE, \theta) \ar[rd]^{(\Gr _{N})^{**}}	&& (V_0, \nabla _0)\ar[rd]^{(\Gr _{S_0})^{**}}&& (V_1, \nabla _1)\ar[rd]^{(\Gr _{S_1})^{**}}&\\
		&	(\cE_0, \theta _0)\ar[ru]^{C^{[-1]}}&&	(\cE_1, \theta_1)\ar[ru]^{C^{[-1]}}&&...\\
}}$$
be  the {canonical Higgs--de Rham sequence} of $(\cE, \theta)$.

By Lemma \ref{L-module-ss-bound}  there exists $\alpha$ such that $\mu _{\max ,L } (\cE _m)-\mu _L (\cE _m)\le \alpha$  for all $m\ge 0$. So by Corollary \ref{Bogomolov's-inequality} there exists some 
constant $C$ such that for every non-negative integer $m$ we have
$$\int _X \Delta (\cE _m) L_2...L_{n-1}\ge C.$$
Lemma \ref{Chern-classes-filtrations} implies that
$$\int _X \Delta (\cE) L_2...L_{n-1} \ge \int _X \Delta (\cE _{0})L_2...L_{n-1}$$
and
$$\int _X \Delta (V_m) L_2...L_{n-1} \ge \int _X \Delta (\cE _{m+1})L_2...L_{n-1}.$$
By Corollary \ref{Chern-classes-Cartier-transform} we have
$$\int _X \Delta (V_m) L_2...L_{n-1}=p^2\int _X \Delta (\cE_m) L_2...L_{n-1}.$$
Therefore 
$$C\le \int _X \Delta (\cE _m) L_2...L_{n-1}\le p^{2m}\int _X \Delta (\cE)L_2...L_{n-1}.$$
Dividing by $p^{2m}$ and passing with $m$ to infinity, we get $\int _X \Delta (\cE) L_2...L_{n-1}\ge 0$.
\end{proof}

\begin{Remark} \label{Bogomolov-inequality-for-connections}
The above theorem holds also for reflexive sheaves with an integrable logarithmic connection.
Indeed, if $(\cE, \nabla)$ is a rank $r\le p$ slope $L$-semistable reflexive sheaf with an integrable 
logarithmic connection and $S^{\bullet}$ is its  Simpson's filtration then 
by the above theorem and Lemma \ref{Chern-classes-filtrations} we have
$$\int_X\Delta (\cE)L_2...L_{n-1}\ge \int_X \Delta ((\Gr _S \cE)^{**})L_2...L_{n-1}\ge 0.$$
\end{Remark}

\section{The Miyaoka--Yau inequality on singular varieties in positive characteristic}\label{Section:Miyaoka-Yau}

In this section we prove the Miyaoka--Yau inequality on some mildly singular varieties in positive characteristic. The ideas are similar to that from \cite{Si} and \cite{La4} but we show full proofs to show where they need additional facts related  to the use of our Chern classes.

\medskip

We fix a log pair $(X,D)$ defined over an algebraically closed field of characteristic $p>0$. We assume that $(X,D)$ is  almost liftable to $W_2(k)$ and it has $F$-liftable singularities in codimension $2$. Let $n=\dim X$ and let us fix a collection  $L=(L_1,...,L_{n-1})$ of nef line bundles on $X$ such that $L_1^2L_2....L_{n-1}>0$. As in  Subsection \ref{Section:Restriction+Bogomolov} we consider a positive open cone $K^+_L\subset N_L (X)$.

The proof of the following proposition is essentially the same as that of \cite[Proposition 4.1]{La4}.

\begin{Proposition}\label{log-Bogomolov's lemma}
	Let $\cL$ be a rank $1$ reflexive sheaf contained in
 $\Omega _X ^{[1]}(\log D)$. Then $c_1(\cL)\not \in K^{+}_L$.
\end{Proposition}

\begin{proof}
	Assume that $c_1(\cL)\in K^{+}_L$ and consider a system of logarithmic Hodge sheaves
	$(\cE:=\cE^1\oplus \cE^0, \theta )$ with $\cE^1=\cL$, $\cE^0=\cO _X$ and
	$\theta : \cE^1\to \cE^0\hat \otimes \Omega_X^{[1]}(\log D)= \Omega_X^{[1]}(\log D)$ given by the
	inclusion. Then $(\cE, \theta)$ is slope $L$-stable since the
	only rank $1$ logarithmic subsystem of Hodge sheaves of $(\cE, \theta)$ is of the form
	$(\cO _X, 0)$. Therefore by Lemma \ref{short-exact-seq} and Theorem \ref{log-Bogomolov-with-lifting}
	we have $$0=4\int_Xc_2(\cE)L_2...L_{n-1}\ge \int _X c_1^2(\cE)L_2...L_{n-1}=c_1(\cL)^2. L_2...L_{n-1},$$
	a contradiction. 
\end{proof}

Similarly as \cite[Theorem 4.4]{La4} one can also get the following theorem generalizing the Miyaoka--Yau inequality in the surface case:

\begin{Theorem}\label{logarithmic-Miyaoka-Yau-with-lift}
Let us assume that $p\ge 3$ and let $\cF\subset \Omega_X^{[1]}(\log D)$ be a rank $2$ reflexive subsheaf with $c_1(\cF)\in \overline{K^+_L}$.  Then 
$$ 3\int_Xc_2(\cF)L_2...L_{n-1}\ge \int_Xc_1(\cF)^2 L_2...L_{n-1} .$$
\end{Theorem}

\begin{proof}
	Let us consider the system of logarithmic Hodge sheaves $(\cE:=\cE^1\oplus \cE^0, \theta )$ given
	by $\cE^1=\cF$, $\cE^0=\cO _X$ and $\theta : \cE^1=\cE \hookrightarrow  \Omega_X^{[1]}(\log D)=\cE^0\hat \otimes  \Omega_X^{[1]}(\log D)$.
	If $(\cE, \theta)$ is slope $L$-semistable, then by Lemma \ref{short-exact-seq} and
	Theorem \ref{log-Bogomolov-with-lifting} we have
	$$
	 3\int_Xc_2(\cF)L_2...L_{n-1}= 3\int_Xc_2(\cE)L_2...L_{n-1}\ge \int_Xc_1(\cE)^2 L_2...L_{n-1}= \int_Xc_1(\cF)^2 L_2...L_{n-1}.$$
	So we can assume that $(\cE, \theta)$ is not slope $L$-semistable. Let
	$(\cE', \theta ')$ be its maximal destabilizing subsystem of logarithmic Hodge subsheaves.
	$(\cE, \theta)$ contains only one saturated rank $1$ system of logarithmic Hodge subsheaves,
	namely $(\cO _X, 0)$. Since this subsystem   does not destabilize $(\cE, \theta)$, the sheaf
	$\cE'$ has rank $2$. Note that $(\cE', \theta ')$ is slope $L$-stable so $c_1(\cM) .L_1...L_{n-1}>0$.
		We can decompose  $\cE'$ into a direct sum $\cO _X\oplus \cM$, where $\cM$ is a saturated
	rank $1$ reflexive sheaf contained in $\cF$. By assumption $(\cE',\theta')$ destabilizes $(\cE, \theta)$ so
	$$\mu _L(\cE')=\frac{c_1(\cM). L_1...L_{n-1}}{2}> \mu _ L(\cE)=\frac{c_1(\cF). L_1...L_{n-1}}{3}.$$
	Therefore $(3c_1(\cM)-2c_1(\cF)).L_1...L_{n-1}>0$. If $(3c_1(\cM)-2c_1(\cF))\in K_L^{+}$ then
	$3c_1(\cM)=(3c_1(\cM)-2c_1(\cF))+2c_1(\cF)\in K_L^{+}$, which contradicts Proposition
	\ref{log-Bogomolov's lemma}. This shows that $$(3c_1(\cM)-2c_1(\cF))^2.L_2...L_{n-1}\le 0.$$
	Let us set $\cL:= (\cF/\cM)^{**}$. Then the sequence 
	$$0\to \cM \to \cF \to \cL$$
satisfies assumptions of Lemma \ref{short-exact-seq} and hence 
	$$\int _X \ch _2 (\cF)L_2...L_{n-1} \le \frac{1}{2} c_1(\cM)^2. L_2...L_{n-1}  +
	\frac{1}{2} c_1(\cL)^2. L_2...L_{n-1} ,$$
	which after rewriting gives
	\begin{align*}
		\int_X(3 c_2(\cF)-c_1(\cF)^2) L_2...L_{n-1}&+\frac{3}{4} c_1(\cM)^2. L_2...L_{n-1}\\
		&\ge -\frac{1}{4}(3c_1(\cM)-2c_1(\cF))^2.L_2...L_{n-1}\ge 0.	
	\end{align*}
Since by Proposition \ref{log-Bogomolov's lemma} we  have $ c_1(\cM)^2. L_2...L_{n-1}\le 0$, this implies the required inequality.
\end{proof}

\section{Applications to  characteristic zero} \label{Section:characteristic-0}

Here we show a few applications of our results to study varieties defined in characteristic zero. In particular, we prove Theorems \ref{Main7}, \ref{Main6} and \ref{Main4}.

\medskip

First we recall the following lemma that follows from Lemma 3.19 in the preprint version of  \cite{GKPT2} (note that any normal surface with quotient singularities is klt).

\begin{Lemma}\label{surface-cover}
	Let $X$ be a normal projective surface with at most quotient singularities defined over an algebraically closed field $k$ of characteristic $0$.
	Let $\cE$ be a coherent reflexive $\cO_X$-module. 
	Then there exists a normal projective surface $Y$ and a finite morphism $\pi: Y\to X$ 
	such that  $\pi ^{[*]}\cE$ is locally free. In this case we have
	$$\int_{X}\Delta (\cE)= \frac {1}{\deg \pi} \int_Y\Delta (\pi^{[*]}\cE)$$
	and 
	$$\int_{X}\ch _2 (\cE)= \frac {1}{\deg \pi} \int_Y\ch_2 (\pi^{[*]}\cE).$$
\end{Lemma}

\medskip
From now on we fix the following notation in this section.
Let $X$ be a normal projective variety of dimension $n$ defined over an algebraically closed field $k$ of characteristic $0$. We assume that $X$ has quotient singularities in codimension $2$ and we fix a reduced divisor $D\subset X$ such that the pair $(X,D)$ is log canonical in codimension $2$. For sheaves on such a variety we use Chern classes defined in \cite[5.3]{La-Chern}. They coincide with classical Mumford's $\QQ$-Chern classes considered in \cite[Chapter 10]{Ko} and in \cite[Theorem  3.13]{GKPT2} (see \cite[Remark 5.9]{La-Chern}). 

\subsection{Strong restriction theorems}

Let us fix a collection  $(L_1, ...,L_{n-1})$ of ample line bundles and let
us set  $d=L_1^2L_2...L_{n-1}$.
The proof of the following theorem is based on  a standard spreading out argument.

\begin{Theorem}	\label{restriction-for-Higgs-char-0}
	Let  $(\cE , \theta )$ be a reflexive logarithmic Higgs sheaf of rank $r\ge 2$ on $(X,D)$. 
	Let $m_0$ be a non-negative integer such that $T_X (\log \, D)\otimes L_1^{\otimes m_0}$  is globally generated. Let   $m$ be an integer such that
	$$m>\max \left(\left\lfloor \frac{r-1}{ r}\int_X\Delta (\cE)L_2\dots L_{n-1} +\frac{1}{r(r-1)d}
	\right\rfloor , 2(r-1)m_0^2 \right).$$
	Let $H\in |L_1^{\otimes m}|$  be good for $(X,D)$.
	\begin{enumerate}
		\item If  $(\cE , \theta )$ is slope $(L_1,...,L_{n-1})$-stable 	then the logarithmic Higgs sheaf $(\cE , \theta )|_H$ on $(H,D\cap H)$ is  slope $(L_2|_H,\dots ,L_{n-1}|_H)$-stable.
		\item If  $(\cE , \theta )$ is slope $(L_1,...,L_{n-1})$-semistable 	and restrictions of all quotients of a Jordan--H\"older filtration of $(\cE , \theta )$ to $H$ are torsion free then  $(\cE , \theta )|_H$ is  slope $(L_2|_H,\dots ,L_{n-1}|_H)$-semistable. 
	\end{enumerate}
\end{Theorem}

\begin{proof} 
	We can find a subring $R\subset k$, which is finitely generated over $\ZZ$ and there exists a flat projective morphism $\cX \to S=\Spec R$ with a relative reduced Weil divisor $\cD$ on $\cX/S$ such that $(X, D)\simeq (\cX\times _S\Spec k, \cD \times _S\Spec k)$. We can assume that 
there exist line bundles $\cL_1,..., \cL_n $ on $\cX$ lifting 	
$L_1, ...,L_{n-1}$, a relative logarithmic Higgs sheaf $(\tilde \cE , \tilde \theta : T_{\cX/S}(\log \cD) \otimes _{\cO_X}\tilde \cE \to \tilde \cE)$ lifting   $(\cE , \theta )$ and a relative effective Cartier divisor 
$\cH \in  |\cL_1^{\otimes m}|$ lifting $H$ (in particular, $\cH\to S$ is flat). Shrinking $S$ if necessary we can assume that the following conditions are satisfied:
\begin{enumerate}
\item $S$ is regular,
	\item all fibers of $\cX\to S$ and  $\cH\to S$  are geometrically integral and geometrically normal,
	\item  $\cL_1,..., \cL_n $ are relatively ample,
	\item $T_{\cX /S} (\log \, \cD)\otimes \cL_1^{\otimes m_0}$ is relatively globally generated,
	\item for all closed points $s\in S$  the fiber $\cX_s$ is liftable modulo $W_2(\kappa (s))$ (see the proof of \cite[Theorem 7]{La3}),
		\item  for all geometric points $\bar s$ of $S$, $\cH_{\bar s}$ is good for $(\cX _{\bar s}, \cD _{\bar s})$,
	\item for all geometric points $\bar s$ of $S$ the sheaf $\tilde \cE_{\bar s}$ is a coherent reflexive $\cO_{\cX_{\bar s}}$-module,
	\item a fixed Jordan--H\"older filtration $\cE_{\bullet}$ of $(\cE, \theta)$ extends to a filtration $\tilde \cE_{\bullet}$ of $(\tilde \cE , \tilde \theta)$.
\end{enumerate} 

Following the proof of \cite[Theorem 4.2]{Ma} (see also \cite[Proposition 2.3.1]{HL}), one can see that geometric slope (semi)stability of logarithmic Higgs sheaves is an open condition in flat families. 
It is sufficient to prove that there exists some geometric point $\bar s$ of $S$, the logarithmic Higgs sheaf $(\tilde \cE , \tilde \theta )|_{\cH_{\bar s}}$ on $(\cX _{\bar s},\cH _{\bar s})$ is  slope $(\cL_2|_{\cH _{\bar s}},\dots ,\cL_{n-1}|_{\cH _{\bar s}})$-(semi)stable. Then  the restriction of $(\tilde \cE , \tilde \theta)|_{\cH}$ to the fiber of $\cH\to S$ over the generic geometric point of $S$ is also slope (semi)stable, proving the theorem.

By the above openness of semistability, we can assume that for all geometric points $\bar s$ of $S$ the logarithmic Higgs sheaf $(\tilde \cE _{\bar s}, \tilde \theta _{\bar s})$ is slope (semi)stable. By the same argument the restriction of the  filtration $\tilde \cE_{\bullet}$ to $\cX_{\bar s}$ gives a Jordan--H"older filtration of  $(\tilde \cE _{\bar s}, \tilde \theta _{\bar s})$. Note also that restrictions of quotients of this filtration to $\cH_{\bar s}$ are torsion free for $\bar s$ over an open subset of $S$.
We need to check that there exists some non-empty open subset $U\subset S$ such that for all geometric points $\bar s$ over  closed points of $U$ we have
$$\int_{\cX_{\bar s}}\Delta (\tilde \cE _{{\bar s}})\cL_2|_{\cX_{\bar s}}\dots \cL_{n-1}|_{\cX_{\bar s}}= \int_X\Delta (\cE)L_2\dots L_{n-1}.$$
This is not obvious as $\tilde \cE$ is not locally free and Chern numbers of reflexive sheaves do not remain constant in flat families (see Example \ref{singular-quadric}).
Using Theorems \ref{main1}, \ref{properties-of-ch_2} and \cite[Theorem 5.8]{La-Chern} we can reduce to the surface case. By Lemma \ref{surface-cover} we can find a normal projective surface $Y$ and a finite covering $\pi: Y\to X$ such such that $\pi^{[*]}\cE$  is locally free. Then, shrinking $S$ if necessary, we can find a flat projective morphism $\tilde \pi: \cY \to S$ and a morphism $\cY \to \cX$ lifting $\pi: Y\to X$. We can also assume that all fibers of $g: \cY\to S$ are geometrically integral and geometrically normal. Since $S$ is normal, the schemes $\cX$ and $\cY$ are also normal.
So we can consider  $\tilde \pi^{[*]}\tilde \cE$, which is reflexive on $\cY$.
This sheaf is locally free outside of a closed subscheme $Z\subset \cY$ of codimension $\ge 2$. Since $Z$ does not intersect the generic fiber of $\cY\to S$,  $\tilde \pi^{[*]}\tilde \cE$ is locally free over a non-empty open subset $S'=S\backslash g(Z)\subset S$. Now let us consider a commutative diagram
$$
\xymatrix{
	\cY_s \ar[r]^{j_s}\ar[d]^{\tilde \pi _s}&\cY\ar[d]^{\tilde \pi}\\
\cX_s \ar[r]^{i_s}& \cX. \\
}$$
 Since $j_s^*(\tilde \pi^{[*]}\tilde \cE)$ is locally free for $s\in S'$ we have an induced map $\varphi _s$ that fits into a commutative diagram
 $$
 \xymatrix{
 \tilde \pi_s^{*} i_s^*(\tilde \cE)\ar[r]\ar[d]^{\simeq}& \tilde \pi_s^{[*]} (i_s^*(\tilde \cE))\ar[d]^{\varphi _s}\\
 j_s^*(\tilde \pi^{*}\tilde \cE)\ar[r]&	j_s^*(\tilde \pi^{[*]}\tilde \cE).\\
 }$$
Let us set $U:=\{x\in X: \tilde \cE _x\hbox{ is a free $\cO_{X,x}$-module}\}$.
Since $\varphi_s$ is an isomorphism over a big open subset 
$\cY_s \cap \tilde \pi^{-1}( U)$ 
of $\cY_S$ and $\tilde \pi_s^{[*]} (i_s^*(\tilde \cE))$ is reflexive, $\varphi _s$ is an isomorphism.
So by Theorem \ref{properties-of-ch_2}  and Lemma \ref{surface-cover} we have
\begin{align*}
		\int_{\cX_s}\Delta (\cE _s)=\frac {1}{\deg \tilde \pi_s} \int_{\cY_s}\Delta (\tilde \pi_s^{[*]}\cE _s)=
\frac {1}{\deg \pi} \int_Y\Delta (\pi^{[*]}\cE)= 	\int_X\Delta (\cE)
\end{align*}
as claimed.
Now the required assertion follows by  applying Theorem \ref{restriction-for-Higgs} to fibers over geometric points $\bar s$ of $S$ with large characteristic of the residue field (then $\beta _r (\bar s)\to 0$).
\end{proof}

\medskip

The same argument as above show also that Theorem \ref{Main1} implies
the following strong restriction theorem of Bogomolov's type:

\begin{Theorem}\label{Mehta-Ramanathan}
	Let  $\cE$ be a coherent reflexive $\cO_X$-module of rank $r\ge 2$. Let  $m$ be an integer such that
	$$m>\left\lfloor \frac{r-1}{ r}\int_X\Delta (\cE)L_2\dots L_{n-1} +\frac{1}{r(r-1)d}
\right\rfloor $$
and let $H\in |L_1^{\otimes m}|$  be a normal hypersurface. \begin{enumerate}
	\item 
	If $\cE$ is slope $(L_1,...,L_{n-1})$-stable 	then $\cE|_H$ is  slope $(L_2|_H,\dots ,L_{n-1}|_H)$-stable.
	\item If $\cE$ is slope $(L_1,...,L_{n-1})$-semistable 	and restrictions of all quotients of a Jordan--H\"older filtration of $\cE$ to $H$ are torsion free	then $\cE|_H$ is  slope $(L_2|_H,\dots ,L_{n-1}|_H)$-semistable. 
\end{enumerate}	
\end{Theorem}

\begin{Remark}
Although Theorem \ref{Main1} works for any normal varieties in positive characteristic, it does not seem easy to use a similar spreading out argument to obtain even the usual Mehta--Ramanathan theorem for ample multipolarizations on a general normal projective variety in characteristic zero.
The problem is that the choice of spreading out depends on $m$ as we need to
spread out divisors  $H\in |L_1^{\otimes m}|$. But since Chern numbers of reflexive sheaves are in general not well behaved in families of normal varieties, we cannot choose one $m$ so that Theorem \ref{Main1} works for this fixed $m$ on even one geometric fiber $\cX_{\bar s}$.
However, using Corollary \ref{Bogomolov-restriction} one can show bounds on the maximal destabilizing slope of $\cE|_H$ on any normal variety in terms of numerical invariants of reductions of $\cE$.
\end{Remark}

\subsection{Bogomolov's inequality for logarithmic Higgs sheaves}

We will need an analogue of the first part of Lemma \ref{short-exact-seq} in the characteristic zero case (the analogue of the second part also holds but we will not need it):

\begin{Lemma}\label{short-exact-seq-0}
Let  $L=(L_1,...,L_{n-2})$ be a collection of nef line bundles on $X$.
If $$0\to \cE_1\to \cE \to \cE_2\to 0$$
is a left exact sequence of reflexive sheaves on $X$, which is also right exact 
on some big open subset of $X$  then
$$\int _X \ch _2 (\cE)L_1...L_{n-2} \le \int _X \ch _2 (\cE _1) L_1...L_{n-2}  +\int _X \ch _2 (\cE _2)L_1...L_{n-2} .$$
\end{Lemma}

\begin{proof}
As in the proof of Lemma \ref{short-exact-seq} we first reduce to the surface case. Then by Lemma \ref{surface-cover} there exists a normal projective surface $Y$ and a finite morphism $\pi: Y\to X$ 
such that  $\pi ^{[*]}\cE$,  $\pi ^{[*]}\cE_1$ and $\pi ^{[*]}\cE _2$ are locally free. 
As in the proof of Lemma \ref{short-exact-seq} we see that the  sequence
$$ 0\to\pi ^{[*]} \cE_1\to\pi ^{[*]} \cE \to\pi ^{[*]} \cE_2\to 0 $$
is  left exact on $Y$ and right exact on some big open subset of $Y$. Therefore
$$\chi(Y,\pi ^{[*]} \cE )\le \chi(Y,\pi ^{[*]} \cE _1)+\chi(Y,\pi ^{[*]} \cE _2).$$
Using the Riemann--Roch theorem for locally free sheaves on normal projective surfaces this can be rewritten as
$$\int _Y \ch _2 (\pi ^{[*]} \cE) \le \int _Y \ch _2 (\pi ^{[*]}\cE _1)  +\int _Y \ch _2 (\pi ^{[*]}\cE _2) .$$
Dividing by the degree of $\pi$, we get  the required inequality from the second part of Lemma \ref{surface-cover}.	
\end{proof}

\begin{Theorem} 
Let  $L=(L_1,...,L_{n-1})$ be a collection of nef line bundles on $X$ such that $L_1^2L_2....L_{n-1}>0$.
For any slope $(L_1,...,L_{n-1})$-semistable logarithmic reflexive Higgs sheaf $(\cE, \theta)$ we have
	$$\int_X \Delta (\cE) L_2...L_{n-1}\ge 0.$$
\end{Theorem}

\begin{proof}	
First assume that $L_1,...,L_{n-1}$ are ample. Then by the above theorem we can restrict to the surface case.  Since finite quotients of smooth affine log surface pairs are $F$-liftable in large characteristics (cf. \cite[Lemma 4.21]{Zd}), we can apply Theorem \ref{log-Bogomolov-with-lifting} and an easy spreading out argument.

In general, we reduce to the above case by an argument analogous to that from the proof of \cite[3.6]{La1}. Namely, we fix an ample line bundle $A$ and consider the classes $L_i(t)=c_1(L_i)+tc_1(A)$ in $N_L(X)\otimes \QQ $ for $t\in \QQ_{>0}$. These classes are ample and the Harder--Narasimhan filtration of $(\cE, \theta)$ with respect to $(L_1(t),...,L_{n-1}(t))$ is independent of $t$ for small $t\in \QQ_{>0}$. 
We have an analogue of Lemma \ref{needed-for-boundedness} for  normal projective varieties in characteristic zero that have quotient singularities in codimension $2$ (this follows from Lemma \ref{short-exact-seq-0} in the same way as Lemma \ref{needed-for-boundedness} follows from Lemma \ref{short-exact-seq}).
Applying this result to the above filtration, using the inequality  for ample collections of line bundles and taking the limit as $t\to 0$ gives the required inequality.
\end{proof}

\section{Appendix: inverse Cartier transform after Lan--Sheng--Zuo}\label{Appendix}

In this section we recall the construction of  inverse Cartier transform from Ogus--Vologodsky's correspondence \cite{OV}, following \cite{LSZ2}. For simplicity of notation we consider only the  non-logarithmic case. The logarithmic case is essentially the same.

\subsection{Results of Deligne and Illusie}\label{Deligne-Illusie}

Below we recall the construction from \cite{DI} of a canonical splitting of the Cartier operator that is associated to a fixed lifting of Frobenius of a smooth $F$-liftable variety. 

\medskip
 
Let $k$ be a perfect field of characteristic $p>0$ and let us set $S=\Spec k$ and $\tilde S=\Spec W_2(k)$.
Let $X$ be a smooth $k$-variety with a fixed lifting $\tilde X/ \tilde S$.
Let $X'$ be the fiber product of $X$ over the absolute Frobenius morphism of $S$. Then we have an induced relative Frobenius morphism $F_{X/S}:X\to X'$. Note that $X'$ has a natural lifting $\tilde X'$ to $\tilde S$, which is defined
as the base change of $\tilde X\to \tilde S$ via $\tilde S\to \tilde S$ coming from $\sigma _2: W_2(k)\to W_2(k)$.
Let us assume that $F_{X/S}: X\to X'$ has a lifting $\tilde F_{X/S}: \tilde X\to \tilde X'$ so that we have a commutative diagram
$$\xymatrix{ 
	X\ar[rr]\ar[dd]\ar[rd]^{F_{X/S}}& &\tilde X\ar[dd] \ar[rd]^{\tilde F_{X/S}}&\\
	&X'\ar[rr]\ar[ld]&&\tilde X'\ar[ld]\\
	S\ar[rr]&&\tilde S&\\
}$$ 
Since the map
$\tilde F^*: \tilde F^*_{X/S}\Omega_{\tilde X'/\tilde S}^1\to \Omega_{\tilde X/\tilde S}^1$
vanishes after pulling back to $X$, we can define 
$$\zeta=p^{-1}\tilde F^*: F^*_{X/S}\Omega_{X'/S}^1\to \Omega_{X/S}^1$$
Since $d\zeta=0$ we can consider $\zeta$ as the map of sheaves of abelian groups $F^*_{X/S}\Omega_{X'/S}^1\to Z_{X/S}^1$, where $Z^1_{X/S}$ is the kernel of $d: \Omega ^1_{X/S}\to \Omega ^2_{X/S}$.
Its adjoint 
$\zeta^{\ad}: \Omega_{X'/S}^1\to F_{X/S, *} Z_{X/S}^1$
is $\cO_{X'}$-linear and it splits the composition  
$$ F_{X/S, *} Z_{X/S}^1\longrightarrow \cH^1 (F_{X/S, *}\Omega _{X/S}^{\bullet})\mathop{\longrightarrow}^{C_{X/S}} \Omega_{X'/S}^1$$
of the Cartier operator with the canonical projection.

\subsection{Inverse Cartier transform}

Assume that $X$ is smooth and there exists a global lifting $\tilde X$ of $X$ to $W_2(k)$.

Let $(\cE, \theta)$ be a  Higgs $\cO_{X'}$-modules with a nilpotent Higgs field of level $\le (p-1)$.
We want to construct  an $\cO_X$-module $V$ with an integrable
connection $\nabla$, whose $p$-curvature is nilpotent of level $\le (p-1)$.
The pair  $(V, \nabla)$ will be denoted by $C^{-1}_{\tilde X/W_2(k)}(\cE,\theta)$
and called the \emph{inverse Cartier transform} of $(E, \theta)$. Let us fix a $W_2(k)$-lifting $\tilde X'$ of $X'/k$ and  take a covering $\{\tilde U_{\alpha}\} _{\alpha\in I}$ of $\tilde X$ such that  for each $\alpha \in I$ there exists $\tilde F_{\alpha}: \tilde U_{\alpha}\to \tilde U_{\alpha}'$ lifting 
the relative Frobenius morphism  $F_{\alpha}: U_{\alpha}\to U_{\alpha}'$.
By previous subsection, the lifting  $\tilde F_{\alpha}$ allows us to construct
$ \zeta _{\alpha}=p^{-1}\tilde F_{\alpha}^*: F^*_{\alpha}\Omega_{U'_{\alpha}/k}^1\to \Omega_{U_{\alpha}/k}^1.$
Therefore over each $U_{\alpha}$ we can define $(V_{\alpha}, \nabla_{\alpha})$ by setting
$V_{\alpha}:=F^*_{\alpha}(\cE|_{U_{\alpha}})$
and
$\nabla_{\alpha}:= \nabla_{can}+\zeta_{\alpha}(F^*_{\alpha}\theta|_{U_{\alpha}'}).$
To glue $(V_{\alpha}, \nabla_{\alpha})$ and $(V_{\beta}, \nabla_{\beta})$ over $U_{\alpha\beta}=U_{\alpha}\cap U_{\beta}$ one uses  the following lemma due to Deligne and Illusie (see \cite{DI}):

\begin{Lemma}
	There exist $\cO_{U_{\alpha\beta}'}$-linear maps $h_{\alpha\beta}: \Omega _{U'_{\alpha\beta}}\to (F_{U'_{\alpha\beta}})_* \cO_{\alpha\beta}$ such that
	\begin{enumerate}
		\item  for all $\alpha,\beta$ we have 
		$$\zeta_{\alpha}^{\ad}-\zeta_{\beta}^{\ad }=dh_{\alpha\beta},$$
		\item  for all $\alpha,\beta,\gamma$ we have over $U_{\alpha\beta\gamma}=U_{\alpha}\cap U_{\beta}\cap U_{\gamma}$
		$$h_{\alpha\beta}+h_{\beta\gamma}=h_{\alpha\gamma}.$$
	\end{enumerate}
\end{Lemma}

Let $h_{\alpha\beta}': F_{U'_{\alpha\beta}}^*\Omega _{U'_{\alpha\beta}}\to  \cO_{\alpha\beta}$ be adjoint to $h_{\alpha\beta}$. 
Now we define gluing maps $g_{\alpha\beta}:V_{\alpha}|_{U_{\alpha\beta}}\to V_{\beta}|_{U_{\alpha\beta}}$
using
$h_{\alpha\beta}'(F^*\theta |_{U_{\alpha\beta}}): F^*\cE|_{U_{\alpha\beta}}\to F^*\cE|_{U_{\alpha\beta}}\otimes F^*\Omega_{U'_{\alpha\beta}}\to F^*\cE|_{U_{\alpha\beta}}$
by setting
$$g_{\alpha\beta}:=\exp (h_{\alpha\beta}(F^*\theta |_{U_{\alpha\beta}}))=\sum_{i=0}^{p-1}\frac{(h_{\alpha\beta}(F^*\theta |_{U_{\alpha\beta}}))^i}{i!}.$$
The maps $g_{\alpha\beta}$ allow us to glue $(V_{\alpha}, \nabla_{\alpha})$ and $(V_{\beta}, \nabla_{\beta})$ over $U_{\alpha\beta}$ to a global object $(V, \nabla)\in \Mic_{p-1}(X/k)$.

\section*{Acknowledgements}

The author would like to thank P. Achinger and T. Kawakami for useful conversations.
A part of the paper was written while the author was an External Senior Fellow at Freiburg Institute for Advanced Studies (FRIAS), University of Freiburg, Germany. 
The author would like to thank Stefan Kebekus for his hospitality during the author's stay in FRIAS.

The  author was partially supported by Polish National Centre (NCN) contract numbers 2018/29/B/ST1/01232 and 2021/41/B/ST1/03741. 
The research leading to these results has received funding from the European Union's Horizon 2020 research and innovation programme under the Maria Sk{\l}o\-do\-wska-Curie grant agreement No 754340.

\end{document}